\documentclass[11pt]{article}
\usepackage{amsthm}
\usepackage{amssymb}
\usepackage{amsfonts,amsmath}
\usepackage{graphicx}
\usepackage[Euler]{upgreek}
\usepackage{multicol}
\usepackage{comment}

\theoremstyle{plain}
\newtheorem{theo}{Theorem}[section]
\newtheorem{prop}[theo]{Proposition}
\newtheorem{lemm}[theo]{Lemma}
\newtheorem{coro}[theo]{Corollary}
\newtheorem{defi}[theo]{Definition}

\theoremstyle{remark}
\newtheorem{remark}{Remark}[section]

\theoremstyle{definition}
\newtheorem{example}{Example}[section]

\newcommand{\diag}{\mathop{\mathrm{diag}}}

\newcommand{\Id}{\mathrm{Id}}

\begin{document}

\title{Invariant manifolds for analytic difference equations}

\author{ Rafael~de~la~Llave$^\dag$\
 and H\'{e}ctor E.~Lomel\'{\i}$^\ddag$\ \\
 \smallskip\\
\dag School of Mathematics\\
Georgia Institute of Technology\\
Atlanta, GA 30332-0160 USA\\
 {\tt rafael.delallave@math.gatech.edu}
\smallskip\\
\ddag Department of Mathematics \\
 Instituto Tecnol\'{o}gico Aut\'{o}nomo de M\'{e}xico\\
	 Mexico, DF 01000 \\
	 {\tt lomeli@itam.mx }}

\date{\today}

\maketitle

\begin{abstract}

We use a modification of the \emph{parameterization method} to
study invariant manifolds for difference
equations. We establish existence, regularity,
smooth dependence on parameters and study several singular limits, 
even if the difference equations do not define
a dynamical system.
This method also leads to efficient algorithms that we present with their
implementations.
The manifolds we consider include 
not only the classical
strong stable and unstable manifolds but also manifolds
associated to non-resonant spaces.

When the difference equations are the Euler-Lagrange equations 
of a discrete variational we present 
sharper results.  Note that, if the Legendre condition
fails, the Euler-Lagrange equations can not be treated as a dynamical 
system. If the Legendre condition becomes singular, the dynamical 
system may be singular while the difference equation remains regular. 
We present numerical applications to  several examples in the physics
literature: the Frenkel-Kontorova model with long-range
interactions and the Heisenberg model of spin chains with a perturbation.
We also present extensions to finite differentiable difference equations. 

\end{abstract}


\section{Introduction}
\subsection{Difference equations}
In this paper we generalize some notions
that have played an important role in
dynamics, namely invariant manifolds, to the more general context of 
implicit difference equations. We also present algorithms to compute 
this object and apply them to several problems in the physics literature. 

Consider a smooth manifold $M$ of dimension $d$.
Let $Z:M^{N+1}\to \mathbb{R}^d$ be an analytic function.
The function $Z$ represents the following difference
equation of order $N$
\begin{equation}\label{eq:del}
Z\left(\theta_{k},\theta_{k+1},\ldots,\theta_{k+N}\right)=0.
\end{equation}
The solutions of the difference equation \eqref{eq:del}
are sequences $(\theta_{k})_{k \geq0}$
which satisfy
\eqref{eq:del}, for all values of $k\geq0$.
A recurrence
\begin{equation}
\label{eq:recurr}
\theta_{k+N} = F\left(\theta_{k},\ldots, \theta_{k+N -1} \right)
\end{equation}
is a particular case of \eqref{eq:del} taking
$Z\left(\theta_{k},\theta_{k+1},\ldots,\theta_{k+N}\right) =
F \left(\theta_{k},\ldots, \theta_{k+N-1} \right) - \theta_{k+N}$.
In some situations, we can consider \eqref{eq:del} as defining
$\theta_{k +N} $ in terms of the other variables, so that
one can transform \eqref{eq:del} into a recurrence of
the form \eqref{eq:recurr}. However, there are instances in which
it is impossible to do so and therefore problem \eqref{eq:del} is
more general than \eqref{eq:recurr}.
Recurrences have been studied mainly as dynamical
systems and they have a rich geometric study.

The goal of this paper is to show that some
familiar constructions in the theory of dynamical
systems have rather satisfactory generalizations
in the context of implicit difference equations.
We will use the so-called parameterization method
\cite{Cabre2003i,Cabre2003ii,Cabre2005iii}
to show existence, regularity and smooth dependence
on parameters of these invariant manifolds.
Note that other methods such as the graph transform, which depend
very much on the dynamical formulation do not seem to generalize 
to the context of implicit difference equations. 

The Legendre transform --if it exists-- makes a difference equation into a dynamical system.
In some situations, a family of Legendre transforms exist but
becomes singular. We study some of these singular limits, in which some invariant manifolds
continue to exist through the singularity of this Legendre transformation
and the implicit equation remains smooth, even if the dynamical system does not.

We will also describe and implement some rather efficient numerical
algorithms to compute these objects with high accuracy.
We point out that some of the
algorithms are novel, even in the dynamical system case
since we study not only the classical stable (and strong stable)
manifolds but also the weak stable manifolds as well as singular
limits. We believe that this is interesting because it shows a
new way of approximating stable and unstable manifolds. In order to
show that the method is robust we perform some explicit calculations that
can be found in the code that supplements this paper.

\begin{remark}
We note that the relation
between difference equations in implicit form
\eqref{eq:del} and the explicit recurrence is
similar to the relation of Differential-Algebraic
Equations (DAE) of the form $Z\left(y, y', \ldots, y^{(n)} \right) = 0 $
and explicit Differential equations
$y^{(n)} = F\left(y,\ldots, y^{(n-1)} \right)$.
It seems that the methods presented here can be extended to DAE,
but one needs extra technicalities. We hope to come back to
this question. For more information on DAE, see \cite{kunkel06}.
\end{remark}

\subsection{Some examples and applications}

\subsubsection{Variational principles}
One important source of the problems of the form
\eqref{eq:del}, to which
we pay special attention, is discrete variational
problems (which appear in physics, economics, dynamic programming, etc.).
A mathematical review of these discrete variational problems
can be found in \cite{Veselov91, Gole01}.

Let $S:M^{N+1}\to\mathbb{R}$ be analytic. When studying variational problems
one is interested --among other things--
in solutions to Euler-Lagrange equations $Z\left(\theta_{k},\ldots,\theta_{k+2N}\right)=0$, with
\begin{equation}\label{eq:Euler-Lagrange}
Z(\theta_0,\ldots,\theta_{2N})\equiv \sum_{j=0}^N\partial_jS\left(\theta_{N-j},\theta_{N-j+1},\ldots,\theta_{2N-j}\right)
=0,
\end{equation}
which are of the form \eqref{eq:del} and of order $2N$.
The Euler-Lagrange equations appear when one finds
critical points of the formal variational principle
based on
$
\mathcal{J}(\theta) \equiv \sum_{k} S\left(\theta_{k},\theta_{k+1},\ldots,\theta_{k+N}\right)
$. Indeed, \eqref{eq:Euler-Lagrange} is formally just
$\frac{\partial}{\partial \theta_{k}} \mathcal{J}(\theta)=0$,
for all $k$.

As we mentioned, there are
well known conditions (Legendre conditions) which
allow to transform the Euler-Lagrange equations
into recurrences, often called discrete Hamilton equations
or twist maps.
In situations where Legendre conditions fail --or become
singular-- the Euler-Lagrange equations cannot be transformed
into Hamiltonian equations. See the examples below
and those treated in more detail in Section \ref{sec:numerics}.

Even if we will not explore it in detail, we
note that the invariant manifolds considered here
have applications to dynamic programming. We plan to come
back to these issues in future investigations.

\subsubsection{The Heisenberg XY model of magnetism}
There are many examples of Lagrangian systems that can not be transformed in
dynamical systems.
Consider, for instance, the Heisenberg XY model. In this model,
one is interested in solutions of
\begin{equation}\label{eq:xy}
\sin (\theta_{k+1}-\theta_{k} ) + \sin (\theta_{k-1} - \theta_{k})
- \varepsilon\sin \theta_{k} =0,
\end{equation}
where each $\theta_{k}$ is an angle that represents the spin state
of a particle in position $k\in\mathbb{Z}$.
The parameter $\varepsilon$ corresponds to the strength of an
external magnetic field.
The equation \eqref{eq:xy} appears as the Euler-Lagrange equation of
an energy functional
$\mathcal{J}(\theta) = \sum_{k} \cos(\theta_{k+1} - \theta_{k}) + \varepsilon
\cos(\theta_{k})$ and has Lagrangian function
$S(\theta_0,\theta_1)=\cos(\theta_{1} - \theta_{0}) + \varepsilon
\cos(\theta_{0})$.

The dynamical interpretation of
\eqref{eq:xy} is problematic because in order to
get $\theta_{k+1}$ in terms of $\theta_{k}$ and $\theta_{k-1}$,
we need to have
\[
\left|\sin(\theta_{k}-\theta_{k-1})+\varepsilon\sin \theta_{k}\right|<1.
\]
This condition is not invariant under the dynamics
and having it for one value of $k$ does not guarantee to have it for others.
However, $\theta_{k}\equiv0$ is a solution of
\eqref{eq:xy}.
In fact, there could exist many solutions,
 defined for $k\geq0$, that converge to the fixed point. So, it is interesting
and useful to identify these solutions and understand their geometry.
We will present algorithms illustrating our general results in
this model in Section~\ref{XYmodel}.

\subsubsection{The Frenkel-Kontorova model
with non-nearest interaction}
\label{sec:efk}
The Frenkel-Kontorova model was introduced to describe dislocations
\cite{Fre-Kon-39},
but it has also found interest in the description of deposition
\cite{Aubry83,Bra-Kiv-04}.
One simplified version of the model (more general
versions will be discussed in Section~\ref{FKgeneral})
leads to the
study of solutions of
non-nearest interactions.
For instance, we have
\begin{equation} \label{simpleFK}
\varepsilon\left(\theta_{k+2} + \theta_{k -2} - 2 \theta_{k} \right)
+ \left(\theta_{k+1} + \theta_{k -1} - 2 \theta_{k}\right) + V'(\theta_{k}) = 0.
\end{equation}

Note that for $\varepsilon = 0$, equation \eqref{simpleFK} can be transformed
into a second order recurrence; i.e. a dynamical system in
$2$ dimensions. Indeed, this $2$-D map is the famous \emph{standard map},
\cite{Mather93,Gole01}, whereas for $\varepsilon \ne 0$ -- no matter how
small -- one is lead to a dynamical system in $4$ dimensions.

It turns out that some terms in this $4-$dimensional system
blow up as $\varepsilon\to0$.
Hence, the perturbation introduced by the $\varepsilon$
term is singular in the Hamiltonian formalism. Nevertheless, we observe that the
singularity appears only when we try to get $\theta_{k+2}$ as
a function of $\theta_{k+1}, \theta_{k}, \theta_{k-1}, \theta_{k-2}$.
The equations
\eqref{simpleFK} themselves depend smoothly on parameters.

Indeed, in section~\ref{sec:singular}, we will show that the invariant
manifolds of \eqref{simpleFK} that we construct, are smooth across $\varepsilon = 0$.
We note that in the applications to solid state
physics, the regime of
small $\varepsilon$ is the physically relevant, one
expects that there are interactions with even longer range
which become smaller with the distance \cite{CardaliguetLFM07,Suzuki}.

\subsubsection{Dependence of parameters of invariant manifolds and Melnikov theory}

In many situations, the  solutions of a difference equation change
dramatically when parameters are introduced. 
In the case of dynamical systems 
the transverse intersection of the stable and unstable manifolds is 
associated with chaos, and gave rise to the famous
horseshoe construction of Smale. The Poincar\'{e}-Melnikov
method is a widely used technique for detecting such intersections
starting from a situation when the manifolds agree. 
One can assume that a system has
pair of saddles and a degenerate heteroclinic or saddle connection
between them. The classical Melnikov theory computes the rate at
which the distance between the manifolds changes with a
perturbation. 
In this context, it is important to understand dependence of parameters
in invariant objects that appear in a Lagrangian setting.
In particular, in the $XY$ and $XYZ$ models, dynamics is not longer useful and
a purely variational formalism is needed.
In this paper, we show that the variational theory is
robust and there is smooth dependence on parameters, so our theory
could lead to a variational formulation of Melnikov's theory.
 Previous work in this direction was done in \cite{Tabacman95,Lomeli97}
 in the context of twist maps.

\subsection{Organization of the paper}

The method of analysis is inspired by the
study in \cite{Cabre2003i,Cabre2003ii,Cabre2005iii}
and is organized as follows.
First, in Section~\ref{sec:linear} we will study the linearized
problem and generalize the notion of characteristic polynomial,
spectrum and spectral subspaces.
They give necessary conditions so that one can even consider
solving \eqref{eq:invariance}.

The Main Theorem~\ref{thm:main} is stated in Subsection~\ref{sec:state}.
After we make the choices of
invariant spaces in the linear approximation, we will
show that, provided that these
spaces satisfy some non-resonance conditions --automatically
satisfied by the (strong) stable spaces considered in the
classical literature--
the solutions of the linear problem lead to solutions
of the full problem.

The main tool is some
appropriate implicit function theorems in Banach spaces.
These implicit function theorems some times require that we consider
approximate solutions of order higher than the first.
In Appendix~\ref{sec:banach} we review the theory
of Banach spaces of analytic functions and describe
how it fits in our problem. In Subsection~\ref{sec:IFT} we
 prove the Main Theorem and in
\ref{higherorder}, we consider the systematic
computation of higher order approximations. Besides
being useful in the proof of theorems,
the higher order approximations are the basis
of efficient numerical algorithms
presented in Section~\ref{sec:numerics}. In particular,
in Example~\ref{exa:xxx}, we show that the method can be used
to approximate slow manifolds, even in the presence of
a singular limit.

Some more refined
analysis also leads to the study of singular limits in
Section~\ref{sec:singular}.
Since the main tool is the implicit function theorem, we
can obtain easily smooth dependence on parameters
(see Subsection~\ref{sec:parameters}).

\section{General Setup}\label{sec:setup}
\subsection{Parameterized solutions}
We are interested in extending
the theory of invariant manifolds
associated to a hyperbolic fixed point to
the more general context of difference equations.
The extension of fixed point is clear.

\begin{defi}
If $\theta^*\in M$ satisfies $Z(\theta^*,\ldots,\theta^*)=0$, then 
we will say that $\theta_{k}\equiv \theta^*$ is a fixed point solution.
\end{defi}

The key to the generalization of the invariant manifolds from
dynamical systems to difference equations is to observe that the invariant
manifolds of a dynamical systems
are just manifolds of orbits. This formulation makes sense
in the context of difference equations. Note also that the dynamics
restricted to the invariant manifold is semi-conjugate to
the dynamics in the whole manifold (in the language of
ergodic theory, it is a \emph{factor}).
Hence, we define:

\begin{defi}[Parameterized stable solution]\label{prob:SMP}
Let $\theta^*$ be a fixed point solution to the difference equation \eqref{eq:del}.
Let $\mathcal{D}$ be an
 open disk of $\mathbb{R}^m$ around the origin. We will say that
 a smooth function $P:\mathcal{D}\to M$ is a stable parameterization
 of dimension $m$ with internal dynamics
 $h:\mathcal{D}\to h\left(\mathcal{D}\right)$
when:
\begin{enumerate}
\item $P(0)=\theta^*$.
\item $0$ is an attracting fixed point of $h$.
\item If $z\in \mathcal{D}$ then
\begin{equation}\label{eq:defpar}
Z\left(P\left(h^{k}(z)\right),P\left(h^{k+1}(z)\right),
\ldots,P\left(h^{k+N}(z)\right)\right)=0,
\end{equation}
for all $k\geq0$.
\end{enumerate}
\end{defi}

\begin{remark}
Notice that, in the definition above, if we let $z_0\in\mathcal{D}$ 
and $\theta_{k}=P(h^k(z_0))$, then
the sequence $(\theta_{k})_{k\geq0}$
satisfies the difference equation \eqref{eq:del} and also
 $\theta_{k}\to\theta^*$, as $k\to\infty$.
\end{remark}

In some situations it might be useful to consider a geometric
object associated to the parameterization. As
a consequence of definition \ref{prob:SMP}, we have the following result.

\begin{prop}
Let $\theta^*$ be a fixed point solution and $P:\mathcal{D}\subset\mathbb{R}^m\to M$
a parameterization of dimension $m$ with internal dynamics 
$h:\mathcal{D}\to h\left(\mathcal{D}\right)$. If $P'(0)$ has rank equal to the dimension of
$\mathcal{D}$, then
there exists an open disk
$\widehat{\mathcal{D}}\subset\mathcal{D}$ around the origin such that
 \begin{equation}\label{LSM}
 \mathcal{W}=\{\left(P\left(z\right),P\left(h(z)\right),
\ldots,P\left(h^{N-1}(z)\right)\right):z\in \widehat{\mathcal{D}} \}
\end{equation}
 is an embedded submanifold $\mathcal{W}\hookrightarrow M^N$ of dimension $m$.
 In such case, we will say that a manifold $\mathcal{W}$ as in \eqref{LSM} is a
 local stable manifold with parameterization $P$.
\end{prop}

We note that the parameterization $P$ and the internal dynamics
$h$ satisfy
\begin{equation}\label{eq:invariance}
\left[ \Phi(P) \right](z) :=
Z\left( P(z), P(h(z)), \ldots, P( h^N(z)) \right) = 0,
\end{equation}
together with the normalizations obtained by
a change of origin
\begin{equation} \label{eq:normalization}
P(0) = \theta^*, \quad h(0) = 0.
\end{equation}

The equation \eqref{eq:invariance} is the centerpiece of
our analysis. Using methods of
functional analysis we can show that, under appropriate
conditions, \eqref{eq:invariance}
has solutions. In order to do this,
we will find it convenient to think of
\eqref{eq:invariance} as the equation $\Phi(P) = 0$,
defined on a suitable function space.

The solutions of \eqref{eq:invariance}
 thus produced generalize the familiar
stable and strong stable manifolds but also include some
other invariant manifolds associated to non-resonant spaces
including, in some cases, the \emph{slow manifolds}.

\subsection{An important simplification}
\label{sec:simplification}

If $L : \tilde{\mathcal{D}}\to\mathcal{D}$
is a diffeomorphism and
we define $\tilde{P}=P\circ L$ and $\tilde{h}=L^{-1}\circ h\circ L$, then
$(\tilde P,\tilde h)$ solves \eqref{eq:invariance}
on the domain $\tilde{\mathcal{D}}$.
In other words, the $P$ and $h$ solving \eqref{eq:invariance} is not
unique. Nevertheless, the range of $P$ and $\tilde{P}$ is unique.

One can take advantage of this lack of uniqueness of
\eqref{eq:invariance} to impose extra normalization on the maps
$h$. Recall that, by the theorem of Sternberg \cite{Ste55,Ste58,Ste59},
under non-resonance conditions
on the spectrum,\footnote{This is automatically satisfied for the strong stable
and strong unstable manifolds and includes the classical stable and unstable manifolds.} 
any analytic one dimensional
invariant curve tangent to a contracting eigenvector has dynamics
which is analytically conjugate to the linear one. So that,
we can take $h$ to be linear if we assume non-resonance or
if we deal with $1$-dimensional
manifolds.

In this paper, we will not consider the resonant case. This justifies
that in \eqref{eq:invariance} we do not consider $h$ as part of
the unknowns, since it will be linear.
In the cases considered here, it will suffice
to consider $h$ to be linear, we will write $h(z) = \Lambda z $.
Hence, instead of \eqref{eq:invariance}, we will consider
\begin{equation}\label{eq:invariancelinear}
\left[ \Phi(P) \right](z) :=
Z\left( P(z), P(\Lambda z), \ldots, P( \Lambda^N z ) \right) = 0,
\end{equation}
where $\Lambda$ is a matrix that will be determined from
a linearized equation near the fixed point. In fact, $\Lambda$ is
a matrix that depends on the roots of a polynomial.
The determination of $\Lambda$ and $P'(0)$ will be studied in Section~\ref{sec:linear}.
We will call $\Phi$ the parameterization operator.
The domain of the operator $\Phi$ is
a function space that depends on the analytic
properties of $Z$ and will be studied in detail in Section~\ref{sec:ana}.

In summary, the basic \emph{ansatz} that we propose is that there are solutions
to the difference equation that are of the form
\[
\theta_k=P\left(\Lambda^kz\right),
\]
where $\Lambda$ is determined by the linearization at the fixed point. In
addition, $\Lambda$ and $P$ are chosen so that $\theta_k\to\theta^*$ as $k\to\infty$.
The smoothness of $P$ depends on the smoothness of $Z$.

\subsection{Unstable manifolds}
With the parameterization method we can also study unstable manifolds.

\begin{defi}[Parameterized unstable solution]\label{prob:UMP}
Let $\theta^*$ be a fixed point solution to the difference equation \eqref{eq:del}.
Let $\mathcal{D}$ be an
 open disk of $\mathbb{R}^m$ around the origin. We will say that
 a smooth function $P:\mathcal{D}\to M$ is a unstable parameterization
 of dimension $m$ with internal dynamics
 $h:\mathcal{D}\to h\left(\mathcal{D}\right)$
when:
\begin{enumerate}
\item $P(0)=\theta^*$.
\item $0$ is an attracting fixed point of $h$.
\item If $z\in \mathcal{D}$ then
\begin{equation}\label{eq:defpar2}
Z\left(P\left(h^{k+N}(z)\right),P\left(h^{k+N-1}(z)\right),
\ldots,P\left(h^{k}(z)\right)\right)=0,
\end{equation}
for all $k\geq0$.
\end{enumerate}
\end{defi}
\begin{remark}
As in the stable case, each parameterization produces
sequences that satisfy the original difference equation. In this
case, if we let $z_0\in\mathcal{D}$ and $\theta_{-k}=P(h^k(z_0))$, then
the sequence $(\theta_{k})_{k\leq 0}$
satisfies the difference equation \eqref{eq:del} and also
 $\theta_{k}\to\theta^*$, as $k\to-\infty$.
\end{remark}
This corresponds to an extension of the original difference
$
Z\left(\theta_{k},\ldots,\theta_{k+N}\right)=0,
$
to negative values $k\leq0$. Alternatively, we can write the
difference equation for negative values in terms of
a dual problem
\[
\tilde{Z}\left(\tilde{\theta}_{k},\tilde{\theta}_{k+1},\ldots,\tilde{\theta}_{k+N}\right)=
Z\left(\tilde{\theta}_{k+N},\tilde{\theta}_{k+N-1},\ldots,\tilde{\theta}_{k}\right)=0.
\]
We interpret the new variable as $\tilde{\theta}_{k}=\theta_{N-k}$.
In this way, the unstable case is reduced to the stable case.

\subsection{Explicit examples}

\begin{example}
 Let $\eta>0$. Consider the function $Z:\mathbb{R}^2\to \mathbb{R}$ given by
\[
Z(\theta_0,\theta_1,\theta_2)=\theta_{2}+\theta_{0}-\frac{2 \cosh (\eta )\theta_1 }{\theta_1^2+1}.
\]
The resulting difference equation
$Z(\theta_{k},\theta_{k+1},\theta_{k+2})=0$
\cite{McMillan71} is an
explicit recurrence called the McMillan map. It is \emph{integrable}
in the sense that
$
 J(x,y)=x^2 y^2 +x^2+y^2-2 \cosh (\eta ) xy
$
satisfies $J(\theta_{k},\theta_{k+1})=J(\theta_{k+1},\theta_{k+2})$.
In \cite{DR98}, it was shown that
$
P(z)= 2 \sinh (\eta )z/(z^2+1)
 $
satisfies
\[
Z(P(z),P(\lambda z),P(\lambda^2 z))=P(\lambda^2 z)+P(z)-\frac{2 \cosh (\eta )P(\lambda z) }{[P(\lambda z)]^2+1}\equiv 0,
\]
provided $\lambda=e^{-\eta}$. Therefore $P$
is a parameterized solution with internal dynamics $h(z) = \lambda z$.
Notice that, in the case of explicit recurrences,
 the parameterization
can follow the manifold even if it folds and ceases to be a graph.
\end{example}

\begin{example}
Let $M=(-\infty,1)$. On $M\times M$, we define the
following the Lagrangian
$S(\theta_0,\theta_1)=\frac12\left(\theta_0-f(\theta_1)\right)^2$,
where $f(\theta)=2\theta/(1-\theta)$. The corresponding
Euler-Lagrange difference equation can be written as $Z(\theta_{k},\theta_{k+1},\theta_{k+2})=0$ where
$Z:M^3\to\mathbb{R}$ is the function given by
\begin{equation}\label{eq:exh}
 Z(\theta_0,\theta_1,\theta_2)=\left(f(\theta_1)-\theta_0\right)f'(\theta_1)+\left(\theta_1-f(\theta_2)\right).
\end{equation}
Clearly, $\theta^*=0$ is a fixed point solution.

Let $P(z)=z/(1-z)$. Since $P(z)\in M$, the maximal disk in which $P$ can be
defined is $\mathcal{D}=(-1/2,1/2)$. If we let $\lambda=\frac12$, then
one can check that $P$ satisfies $f(P(\lambda z))=P(z)$, for all $z\in \mathcal{D}$. 
After substitution, we verify that
$Z(P(z),P(\lambda z),P(\lambda^2 z))\equiv 0$. Therefore, $P$ is a stable parameterized solution for
$\theta^*=0$.

However, the difference equation \eqref{eq:exh} is not a dynamical system because it
can not be inverted, unless $f$ is a diffeomorphism of $M$.
In other words, we can find a parameterization
solution even if the system is not dynamic.

\end{example}

\begin{example}
 Suppose that $G:\mathbb{R}^d\to \mathbb{R}^d$ is a map, possible non-invertible,
 with $G(0)=0$. We would like to parameterize the stable and unstable manifolds
of $0$, if they exist. For instance, we can solve the following pair of one-dimensional problems.
 \begin{itemize}
\item Stable manifold problem: if $0<\lambda<1$ is an eigenvalue of
 $DG(0)$, find a function $P$ such that
 \(
 P(\lambda z)=G(P(z)).
 \)
\item Unstable manifold problem: if $\mu>1$ is an eigenvalue of
 $DG(0)$, find a function $P$ such that
 \(
 P(\mu z)=G(P(z)).
 \)
 \end{itemize}
 Both problems can be solved using the parameterization method for
 difference equations. The same ideas can be used for higher dimensional
 objects. For examples, for $d=3$, two-dimensional stable and unstable manifolds
 were found in \cite{james2010}.
 As we will see, our analysis covers not only one-dimensional
 stable and unstable manifolds,
 but also other non-resonant manifolds.
\end{example}

\section{Linear analysis}\label{sec:linear}

Since we are interested in solutions of \eqref{eq:invariancelinear}
it is natural to study the behavior of the linearization.
As it happens in other situations, this will lead to certain
choices. In subsequent sections, we will show that the
once these choices are made (satisfying some mild conditions),
then there is indeed a manifold which agrees with these
constraints.

Suppose that there exist differentiable $P$, and a matrix $\Lambda$ solving \eqref{eq:invariancelinear}.
Taking derivatives of
\eqref{eq:invariance} with respect to $z$, and evaluating
at $z=0$, we get that
 \begin{equation}\label{eq:linearrelation}
 \sum_{i=0}^NB_iV\Lambda^{i}=0,
 \end{equation}
where $B_i= \partial_{i}Z\left(\theta^*,\ldots,\theta^*\right)$ and $V=P'(0)$.
Our first goal is to understand conditions that allow
to solve equation \eqref{eq:linearrelation} which is a necessary condition
for the existence of differentiable solutions for
\eqref{eq:invariancelinear}.

\begin{remark}
The dimensions of the matrices are determined by the type of parametrization that we have.
Notice that each $B_{i}$ is $d\times d$,
$\Lambda$ is $m\times m$ and $V$ is $d\times m$.
\end{remark}

\begin{defi}[Characteristic polynomial]\label{def:charac}
Let $\theta^*$ a fixed point solution. If we let
\[
B_i= \partial_{i}Z\left(\theta^*,\ldots,\theta^*\right),
\]
then the characteristic polynomial of the fixed point is defined to be:
\[ \mathcal{F}(\lambda):=\det\left( \sum_{i=0}^N\lambda^{i}B_{i}\right). \]
If $\lambda$ is a root of $\mathcal{F}$, then we will say that is it is an eigenvalue. 
The set of eigenvalues is the spectrum, denoted by $\sigma\left(\theta^* \right)$, 
of the fixed point. If $\lambda\in \sigma\left(\theta^* \right)$, then any vector 
$v\in\mathbb{C}^d\setminus\{0\}$ that satisfies \[ \sum_{i=0}^N\lambda^{i}B_{i}v=0 \] 
will be called an eigenvector of $\lambda$.
In addition, if $V$ is a $d\times m$ matrix with non-zero columns and $\Lambda$ is a $m\times m$ matrix that
satisfy the linear relation \eqref{eq:linearrelation},
then we will say that the pair $(V,\Lambda)$ is an eigensolution
of dimension $m$.
\end{defi}

\begin{remark}\label{rm:util}
If the difference equation is of the form
\eqref{eq:Euler-Lagrange} and is the Euler-Lagrange
of a variational principle then the
 characteristic polynomial is of the form
 $\mathcal{F}(\lambda)=\lambda^{Nd}\mathcal{L}(\lambda)$ where
 $\mathcal{L}$ is given by
\begin{equation}\label{eq:characlagr}
 \mathcal{L}(\lambda):=\det\left( \sum_{i,j=0}^N\lambda^{j-i}A_{ij}\right),
\end{equation}
and $A_{ij} =\partial_{ij}S\left(\theta^*,\ldots,\theta^*\right)$.
Now, since $A_{ij}^T=A_{ji}$, we have that if $\lambda$ is in the spectrum, then $1/\lambda $
is also in the spectrum. This is a well known symmetry for the spectrum of
symplectic matrices. In \cite{Veselov91} one can find a proof that,
when the Euler-Lagrange equation \eqref{eq:Euler-Lagrange} defines a dynamical system, it
becomes symplectic.
\end{remark}

\begin{remark}
Let $(V,\Lambda)$ be an eigensolution as above.
Suppose that $w$ is an eigenvector of $\Lambda$ with eigenvalue $\lambda$.
Then $v=\Lambda w$ is an eigenvalue with eigenvalue $\lambda$ in the sense of definition \ref{def:charac}.
\end{remark}

\begin{prop}
Let $\{v_1,\ldots,v_m\}$ be a linearly independent set of eigenvectors with distinct eigenvalues
$\lambda_1,\ldots,\lambda_m$.
Let $V$ be the $d\times m$ matrix
$V=\left(v_1\,v_2\,\cdots\, v_m \right) $ and
$\Lambda$ be the $m\times m$ matrix
$\Lambda=\diag(\lambda_1,\ldots,\lambda_m)$.
Then $(V,\Lambda)$ is an eigensolution and
satisfies equation \eqref{eq:linearrelation}.
In addition, if $Q$ is an invertible matrix,
$\tilde V=VQ$ and $\tilde \Lambda=Q^{-1}\Lambda Q$ then
$(\tilde V,\tilde\Lambda)$ is also an eigensolution.
\end{prop}

\begin{remark} Suppose that
$\lambda=\mu+i\,\nu$ is a root of the characteristic polynomial
$\mathcal{F}$ of a fixed point $\theta^*$. Let $v=r+i\,s$ be an
eigenvector of $\lambda$. If we want to avoid the use of complex numbers,
we can use $\lambda$ and $v$ in order to find a solution with dimension
$m=2$. Let $V$ be the $d\times 2$ matrix given by
$V=(r\ s)$ and
\[
\Lambda=\left(
 \begin{array}{cc}
 \mu & -\nu \\
 \nu & \mu
 \end{array}
 \right).
\]
It is easy to see that $(V,\Lambda)$ is an eigensolution of dimension 2.
\end{remark}

\begin{remark}\label{rem:aux}
We notice that $\mathcal{F}(\lambda)$ is always a polynomial of degree at most $2Nd$. If the matrix $B_{0}$ is non-singular then
the degree of $\mathcal{F}(\lambda)$ is exactly $Nd$. If $\lambda$ is an eigenvalue, 
then there is only a finite number of values of $n$ for which $\mathcal{F}(\lambda^n)= 0$. These considerations motivate the following.
\end{remark}

\begin{defi}[Non-singularity condition]\label{def:nsc}
 We will say that a fixed point $\theta^*$ is non-singular
 if the corresponding characteristic polynomial satisfies: $ \mathcal{F}(0)\neq 0$.
\end{defi}

\begin{defi}[Non-resonance condition]\label{def:nrc}
We will say that an eigenvalue
$\lambda\in \sigma\left(\theta^* \right)$ is
non-resonant if
$ \mathcal{F}(\lambda^n)\neq 0$,
for all $n\geq2$.
\end{defi}

More generally, we will consider non-resonant sets of eigenvalues.
In what follows, we will be using the multi-index notation
where, as usual, if $z=(z_1,\ldots,z_m)\in\mathbb{C}^m$
and $\alpha=(\alpha_1,\ldots,\alpha_m)\in\mathbb{Z}_+^m$
 is a multi-index then
 $z^\alpha=z_1^{\alpha_1}z_2^{\alpha_2}\cdots z_m^{\alpha_m}$.
\begin{defi}\label{non-resonant-set}
We will say that $\boldsymbol{\lambda}=(\lambda_1,\ldots,\lambda_m)$ is a
non-resonant vector of eigenvalues if, for all multi-indices $\alpha\in\mathbb{Z}_+^m$,
\begin{enumerate}
\item $\mathcal{F}(\boldsymbol{\lambda}^\alpha)=0$ if $ |\alpha|=1$,
\item $\mathcal{F}(\boldsymbol{\lambda}^\alpha)\neq0$ if $ |\alpha|>1$.
\end{enumerate}
If $\boldsymbol{\lambda}=(\lambda_1,\ldots,\lambda_m)$ is a
non-resonant vector of eigenvalues then we will also say that the
set $\{\lambda_1,\ldots,\lambda_m\}$ is non-resonant.
\end{defi}

\begin{defi}[Hyperbolicity]\label{defi:hyperbolic}
Let $\theta^*$ be a non-singular fixed point solution of an analytic
difference equation $Z$ and suppose that $\mathcal{F}$ is its characteristic polynomial.
 We will say that $\theta^*$ is
hyperbolic if none of the eigenvalues of $\mathcal{F}$
 are on the unit circle.
Similarly, we will say that a vector of eigenvalues
$\boldsymbol{\lambda}=(\lambda_1,\ldots,\lambda_m)$ is
\emph{stable} if $|\lambda_i | < 1$, for all
$i=1,\ldots,m$.
\end{defi}

\begin{remark}
Let $\theta^*$ be a non-singular fixed point solution and $\sigma(\theta^*)$ its
spectrum. In other words, all elements of $\sigma(\theta^*)$
are non-zero. Notice that, even if
the condition of non-resonance involves infinitely many
conditions, for stable sets all except a finite number of
them are automatic.
Suppose that $\boldsymbol{\lambda}=(\lambda_1,\ldots,\lambda_m)$
is a stable vector of eigenvalues.
If $n \in \mathbb{N}$ is
such that \[
\left( \max_{ \lambda_i \in \boldsymbol{\lambda} }| \lambda_i|\right)^n
\le \min_{\lambda \in \sigma(\theta^*)} |\lambda|,
\]
then we have that
$\boldsymbol{\lambda}^{\alpha} $
cannot be in the spectrum if $ |\alpha|\ge n$. So that there
are only a finite number of conditions to check and the set of
non-resonant $\boldsymbol{\lambda}$ is an open-dense, full measure
set among the stable ones.
\end{remark}

All this analysis shows that there are obstructions
to the computation of invariant manifolds that appear
using just the linear approximation.
These obstructions are a generalization
of the observation, in the dynamical systems case,
that the only invariant
manifolds have to have tangent spaces that are invariant under
the linearization.

The goal of the rest of the paper is to show that if we choose
a subset of the spectrum and an invariant subspace, which
also satisfies the non-resonance conditions, then there is
a solution to the parameterization problem. The analysis will also show
that the non-resonance conditions are needed to obtain a general
result.

\section{Existence and analyticity of solutions}\label{sec:ana}

As indicated before, we will first obtain a formal
approximation and then use an implicit function theorem.
The proof of Theorem~\ref{thm:main} is an application of the
 implicit function theorem in a Banach space of analytic functions
 (see Subsection~\ref{sec:IFT}).
This will lead rather quickly to the analytic 
 dependence on parameters, and the
possibility of getting approximate solutions (see Subsection~\ref{higherorder}).

\subsection{Statement of the Main Theorem}\label{sec:state}

\begin{theo}\label{thm:main}
Let $Z$ be an analytic difference equation function with
a fixed point at $0$ and
such that $B_0=\partial_{0}Z(0)$ is a non-singular matrix.
Let $\boldsymbol{\lambda}=(\lambda_1,\ldots,\lambda_m)$
be a stable non-resonant vector of eigenvalues, with corresponding eigenvectors $v_1,\ldots,v_m$.
Let $(V,\Lambda)$ be the corresponding eigensolution of the
linearized problem \eqref{eq:linearrelation}
with $\Lambda = \diag(\lambda_1, \cdots, \lambda_m)$
and $V=(v_1\cdots v_m)$.

Then, there exist an analytic function $P$ such that its
derivative at $z= 0$ is $P'(0)=V$ and satisfies
\eqref{eq:normalization} and \eqref{eq:invariancelinear}.
The solution is unique among the solutions of the equation.
\end{theo}

\begin{remark}
The radius of convergence of the resulting analytic function
is not specified. In the proof of \ref{thm:main}, we will
consider an equivalent result, in which the radius of
convergence is 1, but
the lengths of the vectors $v_1,\ldots,v_m$ are modified.
\end{remark}

\begin{remark}
 If $Z$ is the Euler-Lagrange equation \eqref{eq:Euler-Lagrange}
 that corresponds to a generating function $S$, then
 $B_0=\partial_{0,N}S(0)$ and hence
 the non-singularity condition is $\det(\partial_{0,N}S(0))\neq0$.
\end{remark}

In Section~\ref{sec:singular}, we will weaken the assumption
$\det(B_{0}) \ne 0$, which is tantamount to a local version of
the Legendre condition. This includes, in particular, the extended
Frenkel Kontorova models with singularities.

\subsection{Analyticity of the parameterization operator}
In appendix \ref{sec:banach}, we give a general definition
of analyticity and introduce spaces of analytic functions.
In particular, given $\rho>0$ we define
 \[
A_\rho(\mathbb{C}^\ell,\mathbb{C}^d)=\left\{f:\mathbb{C}^\ell(\rho)\to \mathbb{C}^d\left|f(z)=
\sum_{k=0}^\infty \sum_{|\alpha|=k}z^\alpha\eta_\alpha,
 \|f\|_\rho<\infty\right.\right\},
\]
where
\[
 \|f\|_{\rho}=\sum_{k=0}^\infty
\left(\sum_{|\alpha|=k}\|\eta_\alpha\|_\infty\right)\rho^k,
\]
$\|\cdot\|_\infty$ is the uniform norm \eqref{eq:uniformnorm}, and
$\mathbb{C}^\ell(\rho)=\{z\in\mathbb{C}^\ell:\|z\|_\infty\leq\rho\}$.

Let $d,N \in\mathbb{N}$ and $\ell=d(N+1)$.
We will consider $\mathbb{C}^\ell\simeq \left(\mathbb{C}^d\right)^{N+1}$.
As an standing assumption, the function that
defines the difference equation satisfies
$Z\in A_{\rho}\left(\mathbb{C}^\ell,\mathbb{C}^d\right)$ for some $\rho>0$;
this is, $Z$ is analytic near the origin.

 In this section we will show that the
 operator $\Phi$, defined in \eqref{eq:invariancelinear}, is
an analytic operator defined on spaces of analytic functions in $\mathbb{C}^m$.
Let $\Lambda$ be an $m\times m$ matrix such that,
if $\|z\|_\infty<1$ then $\|\Lambda z\|_\infty<1$.
We define the linear function
$\mathcal{G}:A_1\left(\mathbb{C}^m,\mathbb{C}^d \right)\to A_1\left(\mathbb{C}^m,\mathbb{C}^\ell \right)$ by
\[
\mathcal{G}(P)=\left(P,P\circ\Lambda,\ldots,P\circ\Lambda^N\right).
\]
Besides being linear, the functional $\mathcal{G}$ has the property that
\(
\mathcal{G}\left(A_1^\rho\left(\mathbb{C}^m,\mathbb{C}^d \right)\right) \subset A_1^\rho\left(\mathbb{C}^m,\mathbb{C}^\ell \right),
\)
for all $\rho>0$.
Let
\begin{equation}\label{eqn:domainX}
 X=A_1(\mathbb{C}^m,\mathbb{C}^d)
\end{equation}
and define the open set
\begin{equation}\label{eqn:domainU}
\mathcal{U}=\{f\in X:\|f\|_1<\rho\}.
 \end{equation}
For a correct application of the implicit function theorem, it suffices $\Phi$ to be $C^1$.
However, we next show if $\Phi$ is analytic.
This will give analytic dependence on parameters of the invariant manifolds.

Notice that
$\mathcal{G}\left(\mathcal{U}\right) \subset A_1^\rho\left(\mathbb{C}^m,\mathbb{C}^\ell\right)$.
Since $Z\in A_{\rho}\left(\mathbb{C}^\ell,\mathbb{C}^d\right)$, we can write
the parameterization operator as a composition
of analytic functions:
$\Phi=\mathcal{C}_Z\circ\mathcal{G}$, where $\mathcal{C}_Z$ is the composition operator
$\mathcal{C}_Z(g)= Z \circ g $.
Consequently, using Lemma~\ref{lemabonito}, we get the following:
\begin{prop}
 Let $\mathcal{U}\subset X$ as in \eqref{eqn:domainX} and \eqref{eqn:domainU}. Then the operator $\Phi:\mathcal{U}\to X$ is analytic.
\end{prop}

\subsection{Fr\'{e}chet derivative of the parameterization operator}

Consider the parameterization operator $\Phi$ defined in \eqref{eq:invariancelinear}, in which $\Lambda$ is a diagonal matrix.
As we have seen, this operator is a function $\Phi:\mathcal{U}\to X$, where
$X$ and $\mathcal{U}$ are defined
 in \eqref{eqn:domainX} and \eqref{eqn:domainU} respectively.

We have the following.
\begin{lemm}\label{lem:frechet}
Let $\boldsymbol{\lambda}=(\lambda_1,\ldots,\lambda_m)$
be a vector such that $\|\boldsymbol{\lambda}\|_\infty<1$.
Let $\Phi:\mathcal{U}\to X$ be the parameterization operator defined in
\eqref{eq:invariancelinear} where $\Lambda=\diag(\lambda_1,\ldots,\lambda_m)$ and define
\begin{equation}\label{eq:funT}
T(\lambda):=\sum_{i=0}^N\lambda^{i}B_{i},
\end{equation}
where $B_{i}=\partial_iZ(0,\ldots,0)$.
Let $\varphi\in X$ be of the form
 \(\displaystyle
 \varphi(z)=\sum_{k=0}^\infty\sum_{|\alpha|=k}z^\alpha \varphi_\alpha .
\) Then $D\Phi(0)\varphi\in X$ and
\begin{equation}\label{eq:important-multi}
 [D\Phi(0)\varphi](z)=\sum_{k=0}^\infty\sum_{|\alpha|=k} z^\alpha T(\boldsymbol{\lambda}^\alpha) \varphi_\alpha.
\end{equation}
 \end{lemm}

 \proof
 We notice the conditions on $\boldsymbol{\lambda}$ imply that
 $\varphi\circ\Lambda^i\in X$ for all $i=0,\ldots,N$. We also have that the
 Fr\'{e}chet derivative of $\Phi$  satisfies
 \begin{equation}\label{eqn:hashsqr}
 [D\Phi(0)\varphi](z)=\sum_{i=0}^NB_{i} \,\varphi\left(\Lambda^{i} z\right),
 \end{equation}
 for all $\|z\|_\infty\leq1$.
This implies that $D\Phi(0)\varphi\in X$.
 Clearly, $\left(\Lambda^{i} z\right)^\alpha=z^\alpha \left(\boldsymbol{\lambda}^\alpha\right)^{i}$ and therefore
 \[
 \varphi\left(\Lambda^{i} z\right)=
 \sum_{k=0}^\infty\sum_{|\alpha|=k} z^\alpha \left(\boldsymbol{\lambda}^\alpha\right)^{i} \eta_\alpha.
\]
Combining the last equation with \eqref{eq:funT} and \eqref{eqn:hashsqr}, we get \eqref{eq:important-multi}.
\qed

Let $H$ be the Banach subspace of analytic functions in the unit disk, 
that vanish at the origin along with their first derivatives.
\begin{equation}\label{eq:espacioH}
H=\left\lbrace P^>\in X\left|
P^>(z)=\sum_{k=2}^\infty\sum_{|\alpha|=k} z^\alpha P_\alpha ;\,
\|P^>\|_1=\sum_{k=2}^\infty\sum_{|\alpha|=k} \|P_\alpha \|_\infty<\infty \right.\right\rbrace .
\end{equation}
\begin{remark}
 In the notation of the appendix \ref{sec:banach}, this subspace is just 
$H=\{P\in A_1(\mathbb{C}^m,\mathbb{C}^d):P(0)=0, P'(0)=0\}$.
\end{remark}

Clearly, \eqref{eq:important-multi} implies that $H$ is invariant under $D\Phi(0)$.
In addition, we get the following result.
\begin{lemm}
\label{lem:bonito}
Let $\Lambda=\diag(\lambda_1,\ldots,\lambda_m)$, where
$\boldsymbol{\lambda}=(\lambda_1,\ldots,\lambda_m)$ is a non-resonant stable vector of eigenvalues.
 If $B_{0}$ is a non-singular matrix, then
\begin{enumerate}
 \item $D\Phi(0)$ is invertible in $H$, with bounded inverse.
 \item If $\varphi\in H$ is such that
\(
 \varphi(z)=\sum_{k=2}^\infty\sum_{|\alpha|=k} z^\alpha \varphi_\alpha
\)
 and $\eta=D\Phi(0)^{-1}\varphi$ can be written as
 \(
 \eta(z)=\sum_{k=2}^\infty\sum_{|\alpha|=k} z^\alpha \eta_\alpha,
\)
 then
\(
\eta_\alpha =T(\boldsymbol{\lambda}^\alpha)^{-1}\varphi_\alpha
\), for every multi-index $\alpha$ such that $|\alpha|\geq2$.
\end{enumerate}
\end{lemm}

\proof
 Clearly, as $k\to\infty$ we get that $T(\lambda^k)\to B_{0}$, a matrix
 that is invertible. This implies that there exists a constant
 $C_0>0$ and a radius $0<\delta<1$
 such that, if $|\lambda|<\delta$ then $\|T(\lambda)^{-1}\|\leq C_0$. In particular, if
 $|\boldsymbol{\lambda}^\alpha|<\delta$ then
 $
 \|T(\boldsymbol{\lambda}^\alpha)^{-1}\|\leq C_0.
 $

We know that, since $\boldsymbol{\lambda}$ is non-resonant, the matrices $T(\boldsymbol{\lambda}^\alpha)$
are invertible, whenever $|\alpha|\geq 2$.
Since $\boldsymbol{\lambda}$ is stable there is only a finite number of
elements in the set
\[
\{\boldsymbol{\lambda}^\alpha\in\mathbb{C}:|\alpha|\geq 2, |\boldsymbol{\lambda}^\alpha|\geq\delta\}.
\]
Let $C>0$ be a constant such that $C\geq C_0$ and
\[
C\geq \max
\{\|T(\boldsymbol{\lambda}^\alpha)^{-1}\|:|\alpha|\geq 2, |\boldsymbol{\lambda}^\alpha|\geq\delta\}.
\]
This implies that
\(
\|T(\boldsymbol{\lambda}^\alpha)^{-1}\|\leq C,
\)
for all $|\alpha|\geq2$.

From \eqref{eq:important-multi} we get that $D\Phi(0)$ is injective.
Let $\varphi\in H$ with
\(
 \varphi(z)=\sum_{k=2}^\infty\sum_{|\alpha|=k} z^\alpha \varphi_\alpha.
\)
Define
\(
 \eta(z)=\sum_{k=2}^\infty\sum_{|\alpha|=k} \eta_\alpha z^\alpha,
\)
with
\(
\eta_\alpha = T(\boldsymbol{\lambda}^\alpha)^{-1}\varphi_\alpha
\), for every multi-index $|\alpha|\geq2$.
From \eqref{eq:important-multi}, we have that $D\Phi(0)\eta=\varphi$ and
therefore $D\Phi(0)$ is invertible in $H$. In addition,
\[
\|D\Phi(0)^{-1}\varphi\|_1=
\|\eta\|_1=\sum_{k=2}^\infty\sum_{|\alpha|=k}\|T(\boldsymbol{\lambda}^\alpha)^{-1}\varphi_\alpha\|_\infty\leq
C\sum_{k=2}^\infty\sum_{|\alpha|=k}\|\varphi_\alpha\|_\infty=C\|\varphi\|_1.
\]
We conclude that $D\Phi(0)$ is invertible in $H$ and $\left\|D\Phi(0)^{-1}\right\|\leq C$.
\qed

\subsection{Proof of Theorem~\ref{thm:main}.}\label{sec:IFT}

Let $H$ as in $\eqref{eq:espacioH}$.
When we consider the linear part $V$ of the
parameterization $P$, it is convenient to choose
the scale sufficiently small. Choosing this scale is
tantamount to choosing the radius of convergence of
the solution $P$.
Let $\tau=(\tau_1,\ldots,\tau_m)$, where $\tau_1,\ldots,\tau_m$ represent scaling factors for
 the columns of the linear part.
We will use the notation $\tau\cdot V=V\diag(\tau_1,\ldots,\tau_m)$.

Since any possible solution of \eqref{eq:invariancelinear}
has to match the lower order terms found, it is
natural to consider a new decomposition
$P = \tau\cdot V + P^{>} $ where $P^{>}$ is an analytic function which vanishes 
to order $2$ and depends on the size of the scale $\tau$. Because of the change of variables, 
we can seek for $P^>$ among analytic functions of radius $1$ that vanish to first order,
i.e. $P^>\in H$.

Hence we write the equation \eqref{eq:invariancelinear}
as
\begin{equation}\label{eq:modified}
\Psi( \tau, P^{>}) \equiv \Phi( \tau\cdot V + P^>) = 0.
\end{equation}
It is important to note, since $V$ is
known, that we only need to find the appropriate scale $\tau$, and
the function $P^{>}$.
Furthermore, since $P^>$ vanishes to order $1$, $\Psi(\tau, P^> )$ also vanishes to order $1$.
In other words, we can choose $H$ to be the
codomain of $\Psi$. In addition, the coefficients of $P^>$ are small if
$\tau$ is small.

We notice that there exists an open subset $\mathcal{V}\subset\mathbb{R}^m\times H$
defined by the property that
if $(\tau,P^>)\in\mathcal{V}$ then $\tau\cdot V +P^>\in\mathcal{U}$, the domain of $\Phi$.
In this way, we can restrict the domain of $\Psi$ and consider
$\Psi: \mathcal{V} \rightarrow H $.
It is clear that this is a neighborhood of the origin $(0,0)$ in which the operator
$\Psi$ is defined and is analytic.

Taking the derivative with respect to the second variable, it is clear that
 $D_2 \Psi(0,0)=D\Phi(0)$ and therefore, by Lemma \ref{lem:bonito}, we have
that the operator
\[
D_2 \Psi(0,0): H \rightarrow H
\]
is invertible with bounded inverse.
The implicit function theorem in Banach spaces \cite{haaser1991real} implies that there exists
$\delta>0$ and a function
 $\tau\mapsto P^<_\tau$ such that if $|\tau|<\delta$, then
$\Psi(\tau,P^>_\tau)=0$ and the solution is unique if
we require $\|P^>\|_1<\delta$.

Fix a solution of \eqref{eq:modified} of the form
$(\tau,P^>_\tau)$ such that $\tau$ has positive entries and
$\Psi(\tau,P^>_\tau)\equiv0$. Then $\tilde P(z)=\tau \cdot V\,z+P^>_\tau(z)$ is a solution
of the parameterization problem. By construction, $\tilde P$ is an analytic
function with radius of convergence $1$. If we want to have $V$ as the linear part of
the solution, we modify $\tilde P$ in the following way. Let
$P=\tilde{P}\circ \diag(\tau_1,\ldots,\tau_m)^{-1}$ or
\[
P(z)=Vz+P^>_\tau(\diag(\tau_1,\ldots,\tau_m)^{-1}z).
\]
Using a linear change of variables as in Section \ref{sec:simplification}, we
conclude that $P$ is also a solution of the parameterization problem.
However, the radius of convergence is not longer $1$ but
depends on the choice of $\tau$. If we let
\(
r=\min\{\tau_1,\ldots,\tau_m\},
\)
then $r>0$, and $\|z\|_\infty\leq r$ implies that
$\|\diag(\tau_1,\ldots,\tau_m)^{-1}z\|_\infty\leq1$.
We conclude that $P\in A_r(\mathbb{C}^m,\mathbb{C}^d)$ and
has radius of convergence equal to $r$.
\qed

\subsection{Formal approximations to higher order}
\label{higherorder}

Once $P'(0)=V$ and $\Lambda$
are chosen, the solution $P$ of the parameterization
problem can be approximated with the first terms of the
power series. Due to analyticity, we can write the solution
$P\in A_\rho(\mathbb{C}^{m},\mathbb{C}^{d})$ of the parameterization problem as
a sum of homogeneous polynomials like in \eqref{eqn:hash2}:
\[
P(z)=\sum_{\ell=1}^\infty\sum_{|\alpha|=\ell} z^\alpha P_\alpha,
\]
where $P_\alpha\in\mathbb{C}^d$, and has $P$ radius of convergence $\rho$.

For each multi-index $\alpha\in\mathbb{Z}_+^m$,
we will use the notation $[\cdot]_\alpha$ for the coefficient vector of the
term that corresponds to $z^\alpha$.
Clearly, if $f$ is an analytic function at the origin, then
this coefficient can be written as
$[f]_\alpha=\partial^\alpha f(0)/\alpha!$. In the case of $P$ above, we get that
 $[P]_\alpha=P_\alpha$.

For each $n\in\mathbb{N}$, let $P^{\leq\,n}$ be the polynomial
\[
P^{\leq\,n}
(z)=\sum_{\ell=1}^n\sum_{|\alpha|=\ell} z^\alpha P_\alpha.
\]
It is clear that $\|P-P^{\leq\,n}\|_\rho\to 0$ as $n\to\infty$.
 Notice also that, for any $|\alpha | \leq n$,
\begin{equation} \label{iorder}
\left[\Phi\left(P^{\leq\,n}\right)\right]_\alpha=\left[\Phi\left(P^{\,}\right)\right]_\alpha=0,
\end{equation}
We will describe how to construct the polynomials $P^{\leq\,n}$
recursively, provided that the eigenvalues are non-resonant.

A simple computation shows that $\left[\Phi\left(P^{\leq\,1}\right)\right]_\alpha=0$
for $|\alpha|=1$.
Let $\mathcal{N}(P)=D\Phi(0)P-\Phi(P)$. Then, it turns out that
\(
\left[\mathcal{N}\left(P^{\leq\,n}\right)\right]_\alpha=\left[\mathcal{N}\left(P^{\leq\, n+1}\right)\right]_\alpha,
\)
for all $|\alpha |= n+1$.
Lemma \ref{lem:frechet} implies that
 $[D\Phi(0)P]_\alpha=T(\boldsymbol{\lambda}^\alpha) P_\alpha$, for all $|\alpha|>1$. Therefore, we conclude that
\begin{align*}
 0=\left[\Phi\left(P^{\leq\,n+1}\right)\right]_\alpha&=\left[D\Phi(0)P^{\leq\,n+1}+\mathcal{N}\left(P^{\leq\,n+1}\right)\right]_\alpha\\
 &=\left[D\Phi(0)P^{\leq\,n+1}\right]_\alpha + \left[\mathcal{N}\left(P^{\leq\,n}\right)\right]_\alpha\\
 &=T(\boldsymbol{\lambda}^\alpha) P_\alpha+\left[\mathcal{N}\left(P^{\leq\,n}\right)\right]_\alpha,
\end{align*}
From this we find the expression
\begin{equation}\label{eqn:nnn}
 P_\alpha= T(\boldsymbol{\lambda}^\alpha)^{-1} \left[\mathcal{N}\left(P^{\leq\,n}\right)\right]_\alpha
 =T(\boldsymbol{\lambda}^\alpha)^{-1} \left[\Phi\left(P^{\leq\,n}\right)\right]_\alpha,
\end{equation}
for all $|\alpha|=n+1$. The polynomial $P^{\leq\,n+1}$ can be found from
$P^{\leq\,n}$ and recursion \eqref{eqn:nnn}.

It is important to note that the choice of $P'(0)$ determines
the tangent space to the manifold. Hence $V$ is determined
once we choose this space. On the other hand,
$P'(0)$ is determined only up to the size of its columns.
These multiples will not be too crucial for the mathematical
analysis, but it will be important in the numerical
calculations in Section~\ref{sec:numerics}.

\begin{lemm}\label{lem:allorders}
With the notations above, assume that
$\boldsymbol{\lambda}=(\lambda_1,\ldots,\lambda_m)$
 is a stable non-resonant vector of eigenvalues,
and that \eqref{eq:linearrelation} is satisfied
with $\Lambda = \diag(\lambda_1, \ldots, \lambda_m)$
and for some $V=P'(0)$ of maximal rank.
Then, for every $|\alpha | \ge 2$, we can find a
unique $P_\alpha$ such that \eqref{iorder} holds.
Furthermore, we can make all $P_\alpha$ arbitrarily
small by making the columns of $P'(0)$ sufficiently small.
\end{lemm}

\begin{remark}
The assumption that the matrix $\Lambda$ is diagonal can be eliminated.
Following the discussion in Section~\ref{sec:simplification}, it suffices that
$\Lambda$ is diagonalizable.
\end{remark}

\begin{remark}
As we will see in  Section~\ref{sec:numerics} the proof
in this section can be
turned into an efficient algorithm using the methods of
\emph{``automatic differentiation''} \cite{JorbaZ,
automatic1, automatic2} which allow a fast evaluation of the
coefficients $P_\alpha$, specially in the case that the manifolds are
$1$-dimensional. More details, including an implementation
in examples, are given int Section~\ref{sec:numerics}.
\end{remark} 

\begin{remark} 
Notice that the main theorem is proved by a contraction 
mapping theorem. The formal solutions are indeed an approximate 
solution. Indeed, in practical problems -- see Section~\ref{sec:numerics} --
it is possible to produce solutions that have an error 
comparable to round off.  These bounds can be proved  
rigorously using interval arithmetic. 

Given some bounds on the contraction properties of 
the  operator, one concludes bounds  on the distance 
between the approximate solution and the true solution.
Hence, the proof presented here gives a strategy to lead 
to computer assisted proofs.
\end{remark}

\section{Singular case and dependence on parameters}

\subsection{Singular case}
\label{sec:singular}

In this section, we show how the results can be extended to the case
in which $B_0$ is singular. The key will be to generalize Lemma \ref{lem:bonito}. This can
be done by estimating the singularity of the matrix
$T(\lambda)$ that was defined in \eqref{eq:funT}.

\begin{example}
Consider the Lagrangian function
$S:\mathbb{R}^2\times\mathbb{R}^2\to\mathbb{R}$ given by
\[
S(\theta_0,\theta_1)=-\theta_0^T\left(
 \begin{array}{cc}
 0 & 0 \\
 0 & 1 \\
 \end{array}
 \right)\theta_1
 +\frac12
\theta_0^T\left(
 \begin{array}{cc}
 1 & 1 \\
 1 & 6 \\
 \end{array}
 \right)\theta_0.
\]
Let $Z$ be the difference equation that arises from the Euler-Lagrange equation \eqref{eq:Euler-Lagrange}.
 The point $\theta^*=(0,0)$ gives a fixed point solution. As before, we define
 $T(\lambda)$ as in equation \eqref{eq:funT}.
 Using definition~\eqref{eq:characlagr} in remark \ref{rm:util},
 it is possible to verify that the characteristic polynomial is of the form
$\mathcal{F}(\lambda)=\lambda(-2\lambda+1)(\lambda-2)$.
 Therefore, $\theta^*$ is singular. We notice that the degree is strictly less than the maximum
 $Nd=4$, $\mathcal{F}(0)=0$ and $\lambda$ divides $\mathcal{F}(\lambda)$.
 \end{example}

In order to deal with the singular case, we have the following result, that
will lead to a generalization of Lemma \ref{lem:bonito}.
 \begin{lemm}\label{lem:unpo}
 Let $\theta^*$ be a singular fixed point solution with characteristic polynomial $\mathcal{F}$.
Let $e$ be the greatest integer $e\in\mathbb{Z}_+$ such that $\lambda^e$ divides $\mathcal{F}(\lambda)$, i.e., 
the polynomial $\mathcal{F}$ is of the form $\mathcal{F}(\lambda)=\lambda^eg(\lambda)$,
 where $g$ is a polynomial such that $g(0)\neq0$.
 Let $T$ be as in \eqref{eq:funT} and $\boldsymbol{\lambda}=(\lambda_1,\ldots,\lambda_m)$
 be a stable non-resonant vector of eigenvalues none of which is zero.
 Then there exists a constant $C>0$ such that
\begin{equation}\label{eq:estimate0}
\left\|T(\boldsymbol{\lambda}^\alpha)^{-1}\right\|\leq C\left(\boldsymbol{\lambda}^\alpha\right)^{-e},
\end{equation}
for all multi-indices $|\alpha|\geq 2$.
 \end{lemm}
 \proof
 The inverse $T(\lambda)^{-1}$ is a rational function of the form
 \[
 T(\lambda)^{-1}=\frac1{\lambda^eg(\lambda)}Q(\lambda),
 \]
 where $Q$ is a polynomial matrix and $g(0)\neq0$. Then $\lambda^eT(\lambda)^{-1}$ is also a rational function, but the limit
 $\lim_{\lambda\to0}\lambda^eT(\lambda)^{-1}$ exists.

 As in the proof of Lemma \ref{lem:bonito}, we can argue that
 since $\boldsymbol{\lambda}$ is non-resonant and stable, the matrices $T(\boldsymbol{\lambda}^\alpha)$
 are invertible and
 $\left(\boldsymbol{\lambda}^{\alpha}\right)^eT(\boldsymbol{\lambda}^\alpha)^{-1}$
 are uniformly bounded for all multi-indices $|\alpha|\geq 2$. Therefore,
 there exists a constant $C>0$ such that
$\left\|\left(\boldsymbol{\lambda}^{\alpha}\right)^eT(\boldsymbol{\lambda}^\alpha)^{-1}\right\|\leq C$.
 \qed

Lemma \ref{lem:unpo} tells us that the derivative
$D\Phi(0)^{-1}$ is an operator, but it might not be well
defined for all the elements of $H$. We will introduce a new Banach space.
For each $\mu>0$, let $D(\mu)=\mathbb{C}^m(\mu)=\{z\in\mathbb{C}^m:\|z\|_\infty\leq\mu\}$
be the complex disk around the origin of complex radius $\mu$. Using $D(\mu)$, we define
 \[
H(\mu):=\left\lbrace P^>:D(\mu)\to\mathbb{C}^d
\left| P^>(z)=\sum_{k=2}^\infty\sum_{|\alpha|=k} z^\alpha P_\alpha ;
\,\|P^>\|_\mu=\sum_{k=2}^\infty\sum_{|\alpha|=k} \mu^n\|P_\alpha \|_\infty<\infty \right.\right\rbrace .
\]

\begin{remark}
 Notice that, in the notation of the appendix,
 \[
 H(\mu)=\{P\in A_\mu(\mathbb{C}^m,\mathbb{C}^d):P(0)=0, P'(0)=0\}.
 \]
\end{remark}
If $\mu_1<\mu_2$ then
$D(\mu_1)\subset D(\mu_2)$ and $H(\mu_2)$ can be regarded as
a subspace of $H(\mu_1)$ through the standard inclusion
$H(\mu_2)\hookrightarrow H(\mu_1)$ given by
\[
f\mapsto \left.f\right|_{D(\mu_1)}
.\] In particular, we have that $H=H(1)$ and if
$0<\mu<1$ then we have the inclusion $H\hookrightarrow H(\mu)$.
In that case, we define
 an operator $\Delta_\mu:H(\mu)\to H$ by
 \[
 \Delta_\mu[\varphi](z)=\varphi(\mu z).
 \]
 Clearly, $\Delta_\mu$ is always bounded. Functions in $H(\mu)$ \emph{gain analyticity}
 through $\Delta_\mu$, and the inverse $\Delta_\mu^{-1}$ is also bounded but \emph{destroys analyticity}. Following the proof of
 Lemma \ref{lem:bonito}, as a corollary of Lemma \ref{lem:unpo} we have the following result.
 \begin{coro} Let $\theta^*$ be a fixed point solution and
 $\boldsymbol{\lambda}=(\lambda_1,\ldots,\lambda_m)$
 be a stable non-resonant vector of eigenvalues.
 Let $e$ be the greatest integer $e\in\mathbb{Z}_+$ such that $\lambda^e$ divides $\mathcal{F}(\lambda)$.

 Suppose that $0<\mu\leq
\min\{|\lambda_1|^e,\ldots,|\lambda_m|^e\}$. Then $D\Phi(0)^{-1}$ is a bounded linear operator $D\Phi(0)^{-1}:H\to H(\mu)$.
In addition,
 the composition $\Delta_\mu D\Phi(0)^{-1}:H\to H$ is bounded and
 \begin{equation}\label{eq:estimate1}
 \left\|\Delta_\mu D\Phi(0)^{-1}\right\|<C,
 \end{equation}
 where $C$ is any constant that satisfies
 \eqref{eq:estimate0} in Lemma \ref{lem:unpo}.
 \end{coro}


\subsection{Dependence on parameters}
\label{sec:parameters}

Suppose that
the difference equation depends smoothly on $q$ parameters
that, for simplicity belong to an open set $\mathcal{E}\subset\mathbb{R}^q$ around the origin.
A special difficulty
arises when there is a parameter in which the equation
becomes singular. We would like to regularize the singular limit.

\begin{example}\label{singFK}
Consider the Lagrangian $S$ defined on $\mathbb{R}^3$ and given by
\[
 S(\theta_{0},\theta_{1},\theta_{2})=
\frac\varepsilon2\left(\theta_{2}-\theta_0\right)^2+\frac12\left(\theta_{1}-\theta_0\right)^2+ W(\theta_0),
\]
where $\varepsilon$ is a small parameter.
From this, we get the Euler-Lagrange equation expressed
as \eqref{eq:Euler-Lagrange}.
If $W'(0)=0$, then the point $\theta^*=(0,0)$ is a fixed point solution.
The characteristic polynomial for this point is
\[
\mathcal{F}(\lambda)=
\varepsilon\lambda^{4}+\lambda^{3}-\left(2\varepsilon+2+W''(0)\right)\lambda^{2}+ \lambda +\varepsilon.
\]
According to Definition~\ref{def:nsc}, the fixed point $\theta^*=(0,0)$ is non-singular if $\mathcal{F}(0)\neq0$ and
this happens if and only if $\varepsilon\neq0$.
\end{example}
\bigskip

As before, let $X=A_1(\mathbb{C}^{m},\mathbb{C}^{d})$.
As an assumption, suppose that there is an open set
$\mathcal{U}\subset X$
such that the equation $\Phi:\mathcal{E}\times \mathcal{U}\to\mathbb{C}^d$ is analytically
 defined. Suppose that $\theta^*=0\in\mathcal{U}$ is a fixed point solution
for all parameters.
On $\mathcal{E}\times \mathcal{U}$,
we define the \emph{nonlinearity} at $\theta^*$ as
\[
\mathcal{N}(\varepsilon,P)=\Phi(\varepsilon,P)-D_2\Phi(\varepsilon,\theta^*)P,
\]
for all $(\varepsilon,P)\in\mathcal{E}\times \mathcal{U}$.

\begin{remark}In general,
for each value of $\varepsilon\in\mathcal{E}$, the corresponding characteristic
polynomial $\mathcal{F}_{\varepsilon}(\lambda)$ is of the form
$\mathcal{F}_{\varepsilon}(\lambda)=\lambda^{e(\varepsilon)}g_{\varepsilon}(\lambda)$, where $e(\varepsilon)$
is an integer that depends on $\varepsilon$ and $g_{\varepsilon}(0)\neq0$. The function $e(\varepsilon)$
does not need
to be continuous and this constitutes a potential difficulty. However, the
theorem below only requires that one can find a constant $\mu$ for which
 the matrices $\mu^{|\alpha|}T(\boldsymbol{\lambda}(0)^\alpha)^{-1}$ are uniformly bounded and
 the nonlinearity $\mathcal{N}$ vanishes at higher order.
\end{remark}

The solution of the stable manifold problem is a local issue, so
we can consider $\mathcal{U}$ and $\mathcal{E}$ as
small neighborhoods that not necessarily cover the largest possible domain
for $\Phi$. Now we can prove a more general theorem that includes parameters.

\begin{theo}\label{thm:parameters}
Let $\Phi$, $\theta^*$, $\mathcal{N}$, $\mathcal{E}$ and $\mathcal{U}$
as above. Suppose that
there exists analytic functions $\lambda_i:\mathcal{E}\to\mathbb{C}$, $v_i:\mathcal{E}\to\mathbb{C}^d$, for $i=1,\ldots,m$ and
constants $C,\mu>0$ such that the following conditions are satisfied, for all $\varepsilon\in\mathcal{E}$.
\begin{enumerate}
\item Each $\lambda_i(\varepsilon)$ is a
non-resonant eigenvalue with eigenvector $v_i(\varepsilon)$.
\item $\boldsymbol{\lambda}(\varepsilon)=(\lambda_1(\varepsilon),\ldots,\lambda_m(\varepsilon))$ is a stable
non-resonant vector of eigenvalues.
\item $\left\|T(\boldsymbol{\lambda}(0)^\alpha)^{-1}\right\|<C\mu^{-|\alpha|}$, for all multi-indices such that $|\alpha|\geq2$.
\item There exists an open neighborhood of the origin $\mathcal{U}_0\subset \mathcal{U}$ such that
the operator $\mathcal{R}(\varepsilon,P)=\mathcal{N}\left({\varepsilon},\Delta_{\mu}^{-1}P\right)$
 can be defined as a function $\mathcal{R}:\mathcal{E}\times \mathcal{U}_0\to X$.
\end{enumerate}
For each $\varepsilon\in\mathcal{E}$, let $V(\varepsilon)=(v_1(\varepsilon)\cdots v_m(\varepsilon))$ and
$\Lambda(\varepsilon) = \diag(\lambda_1(\varepsilon), \cdots, \lambda_m(\varepsilon))$.
Then, there exist a function $P_{\varepsilon}(z)$, analytic in $z$ and $\varepsilon$ such that
$P_{\varepsilon}(0)=0$, its
derivative at $z= 0$ is $\partial_zP_{\varepsilon}'(0)=V(\varepsilon)$ and
\[
Z_{\varepsilon}\left( P_{\varepsilon}(z), P_{\varepsilon}(\Lambda(\varepsilon) z), 
\ldots, P_{\varepsilon}( \Lambda(\varepsilon)^N z ) \right) \equiv 0,
\]
for all $z$ in a neighborhood of the origin.
The solution is unique among the solutions of the equation.
\end{theo}

\proof For simplicity, let the fixed point solution be $\theta^*=0$.
We will use the notation $\Phi_{\varepsilon}=\Phi(\varepsilon,\cdot)$.

Let $H$ as in \eqref{eq:espacioH}.
The main step of the proof is to show that there exists $\delta>0$ and an
analytic function $(\tau,\varepsilon)\mapsto Q^>_{\tau,\varepsilon}$ defined for all 
$\|\tau\|_\infty+\|\varepsilon\|_\infty<\delta$ such that
\[
\Phi\left(\varepsilon,\tau\cdot V(\varepsilon)+Q^>_{\tau,\varepsilon}\right)\equiv0,
\] with
$Q^>_{\tau,\varepsilon}\in H$ and the solution is unique provided $\|Q^>_{\tau,\varepsilon}\|_1<\delta$.

Now, since $0\in\mathcal{U}_0$, there exists an open set $\mathcal{V}\subset \mathbb{R}^m\times\mathcal{E}\times H$ 
that contains $(0,0,0)$ such that
if $(\tau,\varepsilon,Q^>)\in\mathcal{V}$ then
\(
\tau\cdot(\mu V(\varepsilon))+Q^>\in\mathcal{U}_0.
\)
From equation \eqref{eq:estimate1}, we get that
$D\Phi(0)^{-1}$ is a bounded linear operator $D\Phi(0)^{-1}:H\to H(\mu)$ and
$\Delta_\mu D\Phi_0(0)^{-1}:H\to H$ is uniformly bounded.

Let $\mathcal{R}(\varepsilon,P)=\mathcal{N}\left({\varepsilon},\Delta_{\mu}^{-1}P\right)$.
Notice that, if $(\tau,\varepsilon,Q^>)\in\mathcal{V}$ then $\mathcal{R}({\varepsilon},\tau\cdot(\mu V(\varepsilon))+Q^>)\in H$. 
Let $\Psi:\mathcal{V}\to H$ be the operator defined by
\[
\Psi(\tau,\varepsilon,Q^>)= D\Phi_{0}(0)\Delta_\mu^{-1}Q^>+\mathcal{R}({\varepsilon},\tau\cdot(\mu V(\varepsilon))+Q^>).
\]
We notice that $\Psi(0,0,0)=0$ and $D_3\Psi(0,0,0)=D\Phi_{0}(0)\Delta_\mu^{-1}$, that by construction
is invertible with bounded inverse.
Using the implicit function theorem of Banach spaces,
 we can find $\delta>0$ and a function $(\tau,\varepsilon)\mapsto Q^>_{\tau,\varepsilon}\in H$ 
defined for $\|\tau\|_\infty+\|\varepsilon\|_\infty<\delta$ such that
$\Psi(\tau,\varepsilon,Q^>_{\tau,\varepsilon})=0$ and the solution is unique if $\|Q^>_{\tau,\varepsilon}\|<\delta$. 
For each such $(\tau,\varepsilon)$, define
$P^>_{\tau,\varepsilon}=\Delta_\mu^{-1}Q^>_{\tau,\varepsilon}$.
This implies that
\[
D\Phi_{\varepsilon}(0)P^>_{\tau,\varepsilon}+\mathcal{N}(\varepsilon,\tau\cdot V(\varepsilon) +P^>_{\tau,\varepsilon})\equiv 0.
\]
Since $D\Phi_{\varepsilon}(0)V(\varepsilon)\equiv 0$, we have that
$P_{\tau,\varepsilon}=\tau\cdot V(\varepsilon) +P^>_{\tau,\varepsilon}$ is a solution and satisfies
\[
\Phi_{\varepsilon}(P_{\tau,\varepsilon})=\Phi_{\varepsilon}(\tau\cdot V(\varepsilon) +P^>_{\tau,\varepsilon})\equiv 0.
\]
The proof is finished as in the proof of Theorem \ref{thm:main}, with
a change of variables.
Suppose that $P^>_{\tau,\varepsilon}$ is a solution such that the vector $\tau=(\tau_1,\ldots,\tau_m)$ satisfies
$\|\tau\|_\infty<r$ and its entries are positive. Let
\[
P_{\varepsilon}=V(\varepsilon)+P^>_{\tau,\varepsilon}\circ\diag(\tau_1,\ldots,\tau_m)^{-1}.
\]
Then $P_{\varepsilon}(0)=0$, $\partial_zP_{\varepsilon}'(0)=V(\varepsilon)$, and
has a radius of convergence $r$, where $r$ is given by
\[
r=\min\{\tau_1,\ldots,\tau_m\}.
\]
\qed

\section{Examples of numerical algorithms}
\label{sec:numerics}

In this section we describe efficient numerical methods to compute
the parameterizations $P$ described in the previous sections.
We will present algorithms for
\begin{enumerate}
\item[\textbf{A.}] Standard map model with several harmonics.
\item[\textbf{B.}] Frenkel-Kontorova models with long range interactions.\footnote{We will have the C code available in the web.}
\item[\textbf{C.}] Heisenberg $XY$ models.
\item[\textbf{D.}] Invariant manifolds in Froeschl\'{e} maps
\end{enumerate}

For simplicity, the invariant manifolds are $1-$dimensional.
System \textbf{A} is a twist map, \textbf{B} contains a singular limit,
\textbf{C} does not define a map. In \textbf{D} is a $4-$dimensional and
we study both strong stable and slow invariant manifolds.


\subsection{Standard map model with $K-$harmonics}
\label{Chirikov-k}

Let $C_1,\ldots,C_K$ be given numbers. Consider the Lagrangian
$S:\mathbb{R}^2\to\mathbb{R}$ given by
$S(\theta_0,\theta_1)=\frac12(\theta_0-\theta_1)^2+W(\theta_0)$, with
\[
W(\theta)=-\sum_{j=0}^K \frac{C_j}j \cos (j\theta).
\]
The the corresponding
Euler-Lagrange difference equation can be written as $Z(\theta_0,\theta_1,\theta_2)\equiv0$, where
$Z:\mathbb{R}^3\to\mathbb{R}$ is the function given by
\(
Z(\theta_0,\theta_1,\theta_2)=\theta_2-2\theta_1+\theta_0-W'(\theta_1).
\)

Many authors have treated the original ($K=1$) standard map, also known as Chirikov model.
The bi-harmonic model ($K=2$) was studied by \cite{baesens94,calleja06}.
The parameterization problem to solve is
\begin{equation}\label{FKmap}
P(\lambda^{2} z) - 2P(\lambda z)+P(z)
- \sum_{j=0}^K C_j \sin (jP(\lambda z)) =0,
\end{equation}
where $\lambda$ solves
\(
\mathcal{F}(\lambda)=\lambda^2 +\left(- 2 + \sum_{j=0}^K jC_j\right)\lambda +1=0,
\)
and $|\lambda|<1$.
Sometimes, it is more convenient to write
 \eqref{FKmap} as
\begin{equation}\label{FKmap2}
P(\lambda z) - 2P( z)+P(\lambda^{-1}z)
- \sum_{j=0}^K C_j \sin (jP(z)) =0.
\end{equation}
 We solve \eqref{FKmap2} by equating coefficients
of like terms.
The orders 0 and 1 are special.
Clearly, we can take $P_0=0$.
This corresponds to choosing the fixed point $\theta^*=0$.
Equating terms of first order in \eqref{FKmap2}, we obtain:
\begin{equation}\label{eq:firstorder}
\left( \lambda + \lambda^{-1} - 2 + \sum_{j=0}^K j\,C_j\right)P_1 =0.
\end{equation}
Since we choose $\lambda$ so that the term in parenthesis vanishes, we obtain
that $P_1$ is arbitrary. Once we have chosen $\lambda$ so that it solves the
quadratic equation, any $P_1$ will lead to a solution of \eqref{eq:firstorder}.

Any choice of $P_1$ is equivalent from the mathematical point of view
 as they correspond to the choice of scale of the parameterization.
 However, from the numerical point of view it is convenient to choose $P_1$ in
such a way that the subsequent coefficients have comparable sizes so that
the round of error is minimized.
In practice, to find a good choice of $P_1$ we perform a trial run of low
order which gives an idea of the exponential growth or (decay) of the
coefficients $P_n$ and then fix $P_1$ so that the coefficients
$P_n$ neither grow nor decay too much.

The equation for $\lambda$ is quadratic.
The product of its roots is 1 so, when
\[
\left|-2 +\sum_{j=1}^K j\,C_j\right|>2,
\]
we get two roots $\lambda_1$ and $\lambda_2$
of the characteristic polynomial
$\mathcal{F}$ such that $|\lambda_1| < 1$, $|\lambda_2| >1$. We choose
the stable eigenvalue $\lambda_1$.

Since $\lambda^n + \lambda^{-n} - 2 + \sum_{j=1}^K j\,C_j \ne0$
($\lambda^n$ is not a root of $\mathcal{F}$), we get
\begin{equation}\label{eq:solutionn}
P_n = \left( \lambda^n + \lambda^{-n} -2 + \sum_{j=1}^K j\,C_j \right)^{-1}
\left[ \sum_{j=1}^K C_j \sin \left(j P^{\leq\,(n-1)}\right)\right]_n.
\end{equation}
Note that the right hand side can be evaluated if we know $P^{\leq\,(n-1)}$ and hence
we can recursively compute $P_n$. In each step, the coefficients can be found using algorithms
explained below which will be also used in other sections.

\subsection{Efficient evaluation of trigonometric functions}
\label{sec:efficient}
Given a series, $P(z) = \sum_{n=0}^\infty P_n z^n$, we often want to compute
the power series expansions of $\sin (P(z))$ and $\cos (P(z))$.
The following algorithm is taken from \cite{Knuth97}.
Denote $S(z) = \sin (P(z))$ and $C(z) = \cos (P(z))$.
Then, we have
\begin{equation}\label{eq:knuthtrick}
S^\prime (z) = C(z) P^\prime (z),\qquad
C^\prime (z) = -S (z) P^\prime (z).
\end{equation}

Suppose that we can write these functions as
$S(z) = \sum_{n=0}^\infty S_n z^n$ and
$C(z) = \sum_{n=0}^\infty C_n z^n$.
For each $n\in\mathbb{N}$, we will denote
$S^{\leq\,n}(z) = \sum_{k=0}^n S_{k} z^k$,
$C^{\leq\,n}(z) = \sum_{k=0}^n C_{k} z^k$ and
$P^{\leq\,n}(z) = \sum_{k=0}^n P_{k} z^k$. Also,
$[\cdot]_n$ will represent the coefficient of order $n$ of
an analytic function.

Equating terms of order $n$ in \eqref{eq:knuthtrick}, we obtain
\begin{align}\label{recursion}
(n+1) S_{n+1} =& \quad\sum_{j=0}^n C_{n-j} (j+1) P_{j+1},\\
(n+1) C_{n+1} =& -\sum_{j=0}^n S_{n-j} (j+1) P_{j+1}.\nonumber
\end{align}
The recursion \eqref{recursion} allows to compute
the pair of coefficients $S_{n+1}$, $C_{n+1}$
provided that we know the coefficients $S_0,\ldots,S_n$ and $C_0,\ldots,C_n$.
We note that, obviously $S_0$, $C_0$ are straightforward to compute.
From this, we also make the obvious observation that
\begin{align}\label{observation}
S_{n+1} & = \quad C_0 P_{n+1} +\frac1{n+1} \left[\cos\left(P^{\leq\, n}\right)
Q^{\leq\,n}\right]_{n+1},\\
C_{n+1} & = -S_0 P_{n+1} - \frac1{n+1}\left[\sin\left(P^{\leq\, n}\right)
Q^{\leq\,n}\right]_{n+1},\nonumber
\end{align}
where $Q^{\leq\,n}(z)=\sum_{k=0}^n (k+1)P_{k+1} z^k$. In particular,
if $P_0=0$, then we conclude that
\begin{align*}
 \left[S\right]_{n+1}&=P_{n+1}+\left[\sin\left(P^{\leq\,n}\right)\right]_{n+1},\\
\left[C\right]_{n+1}&=\left[\cos\left(P^{\leq\,n}\right)\right]_{n+1}.
\end{align*}

This recursion allows us to get the expansion to order $n$ of $\sin (P(z))$ and $\cos (P(z))$,
given the expansion of $P$ to order $n$.
Furthermore, we observe that if we change $P_n$ --the
coefficient of order $n$ of $P$-- this only affects the coefficients of $\sin (P(z))$ and $\cos (P(z))$ of order $n$ or higher.

The practical arrangement of the calculation of the coefficients
in the standard map with $K$ harmonics is to keep different polynomials
$S_\ell^{\leq\,n}$ and $C_\ell^{\leq\,n}$ that correspond to the series expansions of $\sin(\ell P)$ and $\cos(\ell P)$ up to order $n$.
If the polynomial $P$ is computed to order $n-1$ and the $S_\ell^{\leq\,n}$ and $C_\ell^{\leq\,n}$
corresponding to $P^{\leq\,(n-1)}$ are computed up to order~$n$, we
can compute the coefficient $P_n$ using \eqref{eq:solutionn}.
Then, we can compute the corresponding $S_\ell^{\leq\,(n+1)}$ and $C_\ell^{\leq\,(n+1)}$ up to order $n+1$ using \eqref{recursion}.
We note that similar algorithms can be deduced for $e^{P(z)}$, $\log P(z)$, $P(z)^\gamma$ or
indeed the composition of $P$ with any function that solves a simple differential equation.

\subsection{Frenkel-Kontorova model with extended interactions}
\label{FKgeneral}

\subsubsection{Set up}
Consider the Frenkel-Kontorova model with long range interactions. A
particle interacts not only with its nearest neighbors, but with
other neighbors that are far away. Let $N\geq 2$. We consider
the Lagrangian function $S:\mathbb{R}^{N+1}\to\mathbb{R}$ given by
\[
 S(\theta_{0},\ldots,\theta_{N})=
\frac12\sum_{L=1}^N\gamma_{L}\left(\theta_{L}-\theta_0\right)^2+ W(\theta_0).
\]

The corresponding Euler-Lagrange equilibrium equations are expressions of
$2N+1$ variables, that in this case have the form:
 \[ \sum_{L=1}^N\gamma_{L}\left(\theta_{k+L}- 2\,\theta_{k}+ \theta_{k-L}\right) - W'(\theta_{k})=0.
\]
These equations represent a difference equation of order $2N$.

Suppose that $W'(0)=0$. Then the system has a fixed point solution at $0$ and
the corresponding characteristic polynomial is of the form $\mathcal{F}(\lambda)=\lambda^N\mathcal{L}(\lambda)$,
 where
\[
\mathcal{L}(\lambda)=
\sum_{L=1}^N\gamma_{L}\left(\lambda^{L}-2+ \lambda^{-L} \right)- W''(0).
\]
In addition, the one-dimensional parameterization equations of the point can be written as
\[
\sum_{L=1}^N\gamma_{L}\left(P(\lambda^{N+L}z)- 2\,P(\lambda^{N}z)+ P(\lambda^{N-L}z)\right) - W'(P(\lambda^{N}z))=0,
\]
where $\lambda$ is a non-resonant stable root of the characteristic function $\mathcal{L}$.

\begin{remark}
We can simplify the characteristic polynomial above.
Notice that, if we let
$ \omega=(\lambda+\lambda^{-1})/2$,
then
\[
 \frac{\lambda^{L}+\lambda^{-L}}2=\mathcal{T}_{L}(\omega),
\]
where $\mathcal{T}_{L}$ is the $L-$th Tchebychev polynomial.
Let
 $r(\omega)$ be the polynomial of degree $N$ given by
\begin{equation}\label{eq:first}
 r(\omega)=\sum_{L=1}^{N}\gamma_L(\mathcal{T}_{L}(\omega)-1)-\frac12W''(0).
\end{equation}
Then, characteristic polynomial $\mathcal{F}(\lambda)$, can be written as
\(
 \mathcal{F}(\lambda)=2\lambda^N\,r\left( (\lambda+\lambda^{-1})/2\right).
\)
In addition, $\mathcal{F}(\lambda)$ has no zeroes on the unit circle if and only if
$r(\omega)$ has no roots on the segment $[-1,1]\subset \mathbb{C}$. For each root
$\omega$ of $r$, we get a pair of eigenvalues. If $\omega$ is real and $|\omega|>1$, then these
eigenvalues are a pair of real numbers 
\[
\lambda^{s,u}=\omega\pm\sqrt{\omega^2-1}
\]
that satisfy
$0<|\lambda^s|<1<|\lambda^u|$.
\end{remark}

\subsubsection{Singular limit and slow manifolds}

In many situations the long-range
interactions of the particles in the model are small. We could ask the question of what happens in
the limit. It turns out that the system becomes singular and the usual dynamical systems approach fails
to be useful. However, certain stable manifolds persist, as in Theorem~\ref{thm:parameters}.
We illustrate this difficulty with an example.

\begin{example}
Consider a Frenkel-Kontorova equation with $\gamma_1=1$ and $\gamma_2=\varepsilon$.
In this case, the auxiliary polynomial in \eqref{eq:first} is
\[
r(\omega)=\varepsilon(2\omega^2-2)+\omega-\beta,
\]
where $\beta=1+\frac12W''(0)$.

Solving for $\omega(\varepsilon) $ , we get that \[
\omega^\pm(\varepsilon)=\frac{2(\beta+2\varepsilon)}{1\pm\sqrt{1+8\varepsilon(\beta+2\varepsilon)}}.
\]
If $\varepsilon\to 0$ then we have a singular limit. We notice that, as $\varepsilon\to 0$,
the two roots of the polynomial have two different limits $\omega^+(\varepsilon)\to\beta$ and
$\omega^-(\varepsilon)\to\infty$. In terms of the stability of the fixed point, the limit
$\omega^-(\varepsilon)\to\infty$ corresponds to a pair of eigenvalues $\lambda^{s},\lambda^{u}$ that
are very hyperbolic in the sense that $\lambda^{s}\lambda^{u}=1$ and
$\lambda^{s}\to0$ and $\lambda^{u}\to\infty$ as $\varepsilon\to0$.

Fortunately, the other pair of eigenvalues can be continued through the
singularity $\varepsilon=0$. This family is smooth and will be denoted by
$\lambda^s(\varepsilon)$, $\lambda^u(\varepsilon)$. They satisfy
$0<\lambda^s(\varepsilon)<1<\lambda^u(\varepsilon)$, $\lambda^s(\varepsilon)\lambda^u(\varepsilon)=1$ and
$\lambda^s(0)+\lambda^u(0)=2\beta$.
\end{example}

\begin{remark}
In general, if we let $\beta=1+W''(0)/(2\gamma_1)$, then
 there is a family of roots $\omega$ of $r(\omega)$ such that $\omega\to\beta$ as
 $(\gamma_2,\ldots,\gamma_N)\to 0$. It follows that, if $|\beta|>1$ and the coefficients
 $\gamma_2,\ldots,\gamma_N$ are small enough, then the fixed point $\theta^*=0$ is hyperbolic.
 This occurs, for instance, when $0$ is a minimum of the potential $W$, $\gamma_1>0$, and the
 long range interactions are weak.
\end{remark}

We are interested in the persistence of slow manifolds in the Frenkel-Kontorova model with long-range
interactions. We can consider that the interactions are small. Suppose that we have
$N$ long-range interactions represented by small coefficients
$\gamma_2(\varepsilon),\ldots,\gamma_N(\varepsilon)$ that depend
analytically on the parameter $\varepsilon$.
Assume that $\gamma_2(0)=\cdots=\gamma_N(0)=0$,
$\gamma_N^\prime(0)\neq 0$ and, without loss of generality,
that $\gamma_1(\varepsilon)\equiv 1$.

For each $\varepsilon$, the characteristic polynomial $\mathcal{F}_\varepsilon$ of
the fixed point $\theta^*=0$ is of degree at most $Nd$. From the implicit
function theorem, we can argue that there exists a number $\varepsilon_0$>0 and a
smooth function
$\omega(\varepsilon)$ such that, if $|\varepsilon| \leq\varepsilon_0$ then
$r_\varepsilon(\omega(\varepsilon))\equiv0$, where $\omega(0)=\beta$ and
\[
r_\varepsilon(\omega)=\sum_{L=1}^{N}\gamma_L(\varepsilon)(\mathcal{T}_{L}(\omega)-1)-\frac12W''(0).
\]
All the other roots of $r_\varepsilon(\omega)$ diverge as $\varepsilon\to\infty$.
For each non-singular root $\omega(\varepsilon)$ of $r_\varepsilon$, we get the following stable
eigenvalue
\[
\lambda^{s}(\varepsilon)=\omega(\varepsilon)-\sqrt{\omega(\varepsilon)^2-1}.
\]
The family $\lambda^{s}(\varepsilon)$ can be continued through the singularity
$\varepsilon=0$ and, for all $|\varepsilon| \leq\varepsilon_0$, it corresponds to
a slow manifold, i.e. the invariant manifold with the largest stable eigenvalue.
In fact, $\lambda^{s}(\varepsilon)$ is analytic near $\varepsilon=0$.

For each $|\varepsilon| \leq\varepsilon_0$, there exists a non-negative
integer $e(\varepsilon)$ such that
$\mathcal{F}_\varepsilon(\lambda)=\lambda^{e(\varepsilon)}g_\varepsilon(\lambda)$. By construction,
$e(\varepsilon)=0$ if $\varepsilon\neq0$ and $e(\varepsilon)=N-2$ if $\varepsilon=0$.
Using the notation of Theorem \ref{thm:parameters}, we have that the nonlinearity
of the parameterization operator at $\theta^*$ is
precisely
\[
[\mathcal{N}(\varepsilon,P)](z)=W''(0)P\left(\lambda^{s}(\varepsilon)^Nz\right)-W'\left(P(\lambda^{s}(\varepsilon)^Nz)\right).
\]

Let $\mu$ be a number such that
\[
\lambda^{s}(0)^N<\mu<\lambda^{s}(0)^{N-2}.
\]
Restricting $\varepsilon_0$ further, it is possible to assume that $\mu$ also
satisfies $\lambda^{s}(\varepsilon)^N<\mu<\lambda^{s}(\varepsilon)^{N-2}$,
for all $|\varepsilon| \leq\varepsilon_0$.
Using Lemma~\ref{lem:unpo}, we conclude that there exists a constant $C$ such that
the matrix norm satisfies
\[
\|\lambda^{s}(0)^{(N-2) k}T(\lambda^{s}(0)^k)^{-1}\|\leq C,
\]
 for all $k\geq2$. Furthermore, we have that
\[
\|\mu^{k}T(\lambda^{s}(0))^{-1}\|
\leq C\mu^{k}\lambda^{s}(0)^{-(N-2) k}<C,
\]
for all $k\geq 2$.
Now, the condition on the nonlinearity is that
$\mathcal{R}\left(\varepsilon,P\right)=\mathcal{N}\left(\varepsilon,\Delta_{\mu}^{-1}P\right)\in X$, for all $P\in X$ 
in a neighborhood of the
origin. However, the nonlinearity satisfies
\[
\left[\mathcal{R}\left(\varepsilon,P\right)\right](z)=
\left[\mathcal{N}\left(\varepsilon,\Delta_{\mu}^{-1}P\right)\right](z)=
W''(0)P\left(\mu ^{-1}\lambda^{s}(\varepsilon)^Nz\right)-W'\left(P(\mu ^{-1}\lambda^{s}(\varepsilon)^Nz)\right).
\]
Since $|\mu ^{-1}{\lambda^{s}(\varepsilon)^N}|<1$, we conclude that there exists an open neighborhood $\mathcal{U}_0$
of $0$ in which the nonlinearity can be defined as an operator.
From these considerations we conclude that the Theorem \ref{thm:parameters} applies.
Therefore, there exists a family of analytic solutions $P_\varepsilon$ of the parameterization problem.
This is illustrated in the numerical example that follows.

\subsubsection{Some numerics}
If we consider the $K-$harmonic potential
\(
W(\theta) = -\delta\,\sum_{j=1}^K \frac{C_j}{j}\ \cos (j\theta)
\), then the equilibrium equations are
\begin{equation}\label{eq:invarianceextended}
\sum_{L=1}^N \gamma_L (\theta_{k+L} - 2\theta_{k}+ \theta_{k-L} )
+ \delta\,\sum_{j=1}^K C_j \sin (j\theta_{k}) = 0.
\end{equation}
The parameterization equations of a $1-$dimensional stable manifold
can be written as
\begin{equation}\label{eq:paramFK}
\left[\Phi(P)\right](z)=\sum_{L=1}^N \gamma_L (P( \lambda^L z) -2P(z)+ P(\lambda^{-L} z) )
+\delta\, \sum_{j=1}^K C_j \sin (jP(z)) =0,
\end{equation}
where $\lambda$ is a non-resonant stable eigenvalue. By symmetry, if $P$
parameterizes a stable manifold it also parameterizes an unstable one.
The characteristic polynomial is $\mathcal{F}(\lambda)=\lambda^N\mathcal{L}(\lambda)$,
 where
\[
\mathcal{L}(\lambda ) = \sum_{L=1}^N \gamma_L (\lambda^L + \lambda^{-L} -2)
+ \delta\sum_{j=1}^K j\,C_j.
\]
We will find a stable parameterization $P$  corresponding to the fixed point solution $0$.  Suppose that $P$ is of the form
$P(z)=\sum_{k=0}^\infty z^kP_k$.  We set, therefore,
$P_0=0$.
Also, we chose a solution $\lambda$ of $\mathcal{L}(\lambda) =0$, which
amounts to choosing the stable manifold we want to study, and set $P_1$ so that the numerical error
is minimized.
When $n\ge2$, matching coefficients of $z^n$ in \eqref{eq:paramFK},
we obtain
\begin{equation}\label{eq:extendedn}
\mathcal{L}(\lambda^{n+1}) P_{n+1} +\delta\, \left[\sum_{j=1}^K C_j \sin\left( j P^{\leq\,n}\right)\right]_{n+1} =0.
\end{equation}
For a generic set of values of $\gamma_1,\ldots,\gamma_N$ and
$C_1,\ldots,C_K$, we have that $\mathcal{L}(\lambda^{n+1}) \neq 0$, for all
$n\in\mathbb{N}$. Therefore, we can solve the equation \eqref{eq:extendedn} and
get a non-resonant eigenvalue.

We keep the polynomials $P$, $S^j$, $C^j$ as in Section \ref{sec:efficient} and assume that we know
$P^{\leq\,(n-1)}$ and the $S^j,C^j$ corresponding to $P^{\leq\,(n-1)}$ up
to order~$n$. We use \eqref{eq:extendedn} to compute $P_n$ and then
\eqref{recursion} to compute the $S^j,C^j$ corresponding to $n$ up to
order~$n-1$.
The only difference with the short range case is that, when solving the
recursion, we need to divide by a slightly different
factor.

\begin{figure}[ht]
\begin{center}
\includegraphics[height=3.4in]{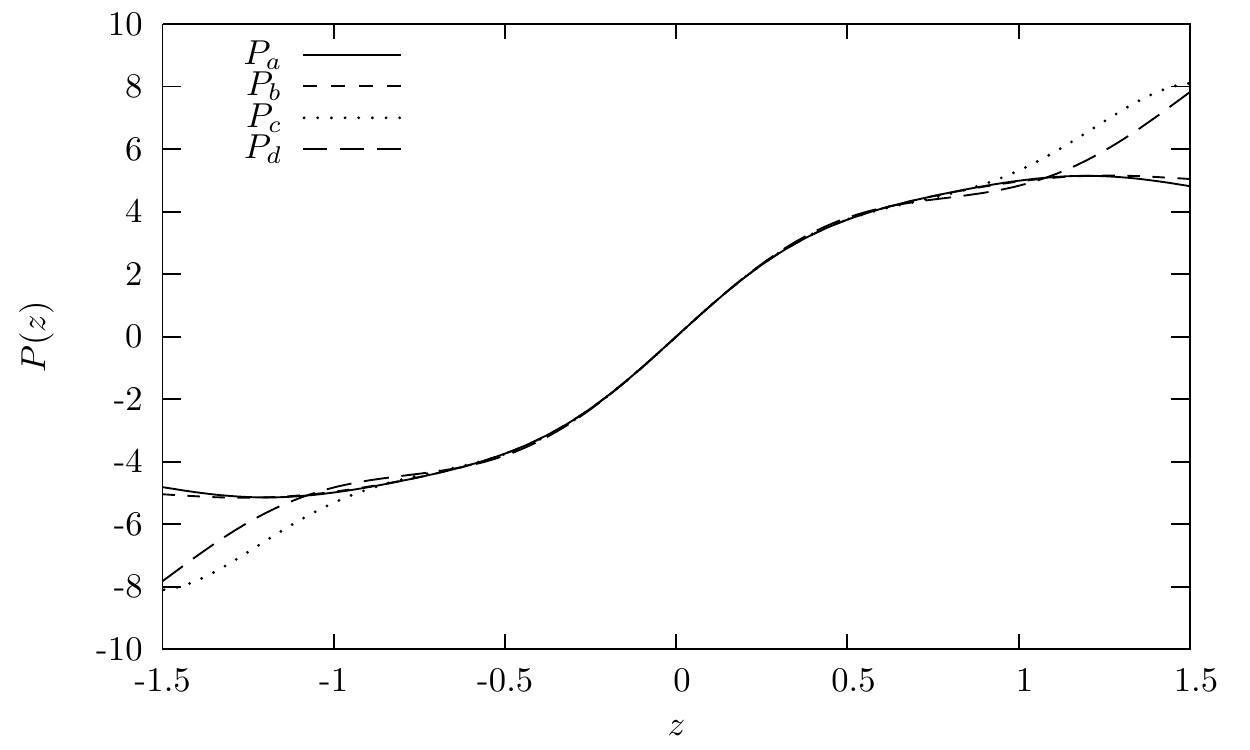}
\caption{Four parameterizations for the Frenkel-Kontorova model
of example \ref{exa:xxx} with parameters given in table \ref{table:FK}.}\label{fig:FK1}
\end{center}
\end{figure}

\begin{figure}[ht!]
\begin{center}
\includegraphics[height=3.4in]{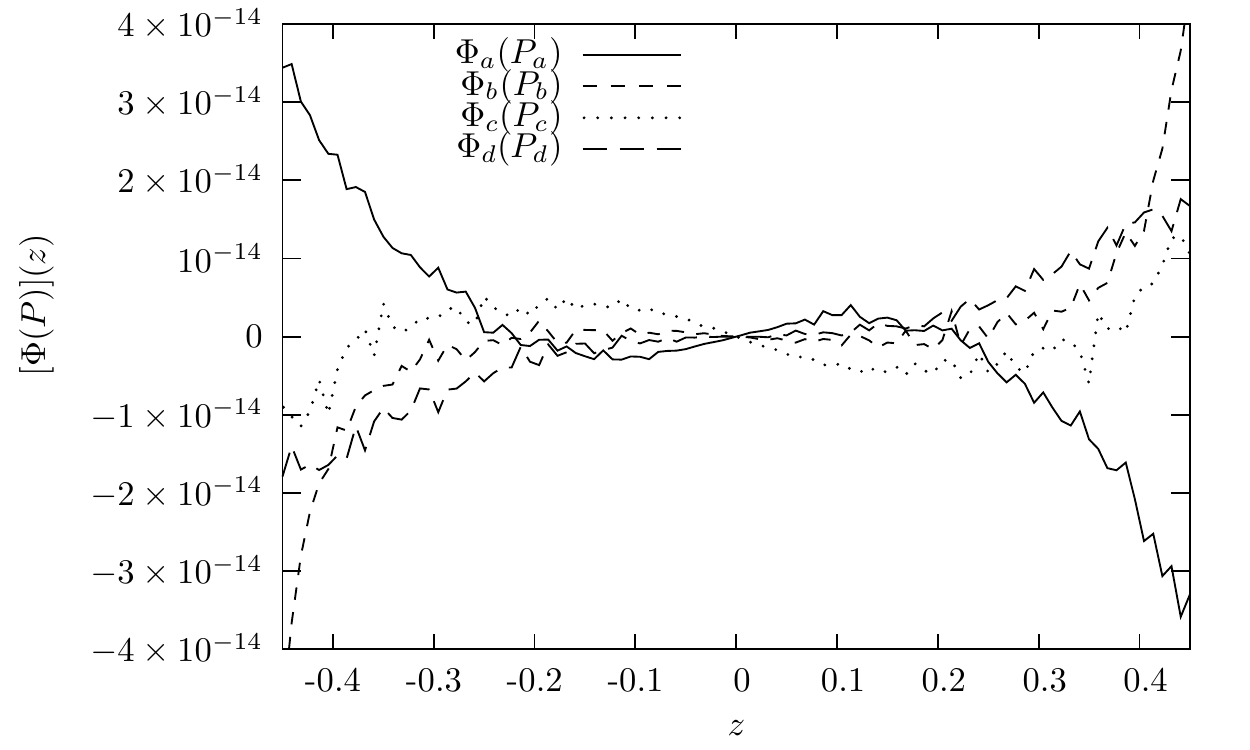}
\caption{Error function $\Phi(P)$ for the approximation to the parameterization solution of
the Frenkel-Kontorova model
of example \ref{exa:xxx} with the parameters given in table \ref{table:FK}.}\label{fig:FK2}
\end{center}
\end{figure}

\begin{figure}[hb!]
\begin{center}
\includegraphics[height=3.4in]{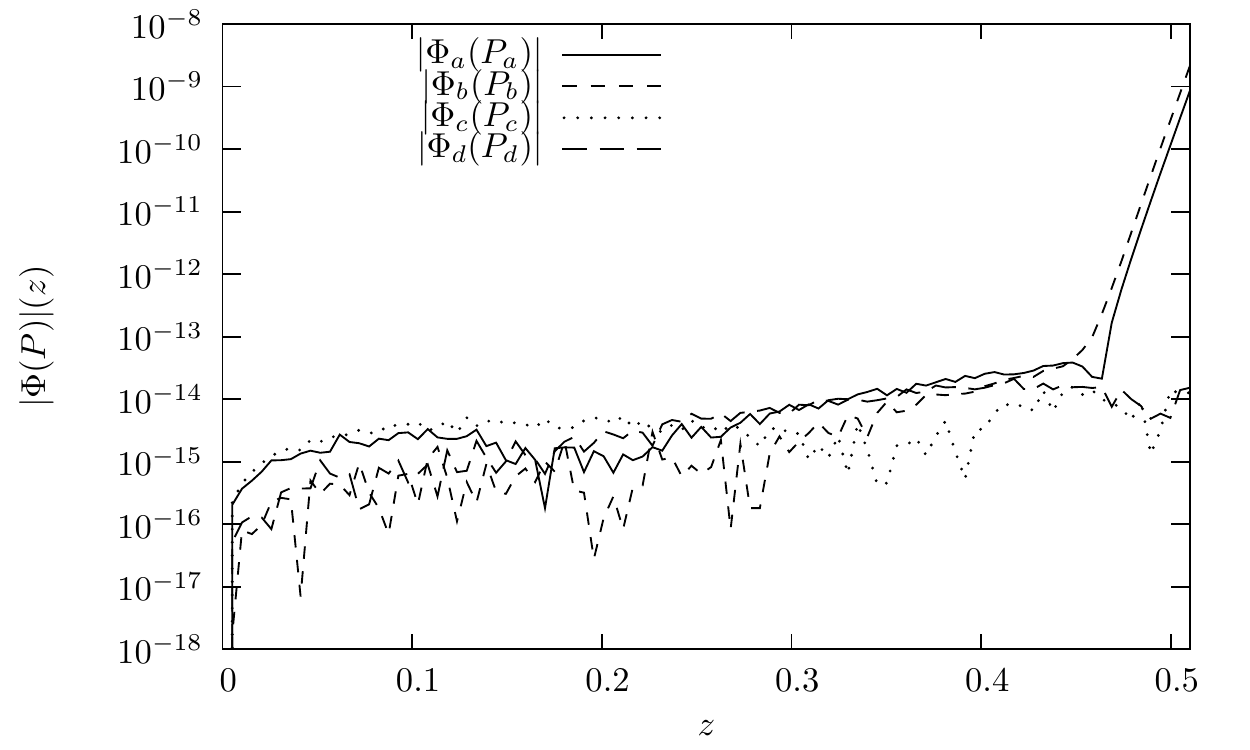}
\caption{These graphs illustrate the absolute value $|\Phi(P)|$ of the error function for the Frenkel-Kontorova model
of example \ref{exa:xxx} with parameters given in table \ref{table:FK}. Logarithmic scale is used in the vertical axis.}\label{fig:FK3}
\end{center}
\end{figure}

\begin{example}\label{exa:xxx}
Consider an specific singular limit.
Let the Frenkel-Kontorova model with $N=3$, $K=1$, $\delta=0.4$, and
$C_1=1$. We  fix the values of $\gamma_1,\gamma_2,
\gamma_3$ in four examples.

Using the algorithm proposed of this section,
we find the first $100$ coefficients of the Taylor series expansion
of the parameterizations corresponding
to the following values. In each case, the computed
eigenvalue $\lambda$ corresponds to the slow manifold.
\begin{table}[ht]
\centering
\begin{tabular}{|c|c|c|c|c|c|c|c|}
 \hline
 Parameterization & $\gamma_1$ & $\gamma_2$ & $\gamma_3$ &$\delta$ & $P'(0)$ & $\lambda$\\
 \hline
 $P_a$ & 1 & 0.1 & 0.00 & 0.4 & 10.0 &  0.592583231399561\\
 $P_b$ & 1 & 0.14 & 0.00  & 0.4 & 10.0 &  0.609158827181520\\
 $P_c$ & 1 & 0.1 & 0.01  &  0.4 & 10.0 & 0.603202338024902\\
  $P_d$ & 1 & 0.1 & 0.03  & 0.4 & 10.0 & 0.621569001269222 \\
 \hline

\end{tabular}
\caption{Parameters for the Frenkel-Kontorova examples.}
\label{table:FK}
\end{table}

The solution of the parameterization problem \eqref{eq:paramFK} is, in fact, a family of
functions that depend on the size of the derivative.  We find uniqueness only when
the derivative $P'(0)$ is fixed. This value is determined so that the coefficients
are of order 1. As input we use the parameters given in table \ref{table:FK} and,
as output, we get the approximation $P^{\leq 100}$.

In the example, parameterizations of the slow manifold are computed.
The dimension of the  problem changes but the method allows the continuation of the solution
through the singularity. In other words, if we regard the difference equation as a dynamical system,
then cases $a$ and $b$ would be maps in $\mathbb{R}^4$ and   cases $c$ and $d$ would be maps
in $\mathbb{R}^6$. This collapse in the dimension is problematic if one uses a dynamical
system point of view, but is manageable when the Lagrangian point of view is used.

We notice that the four parameterization functions
$P_a$, $P_b$, $P_c$, and $P_d$ are similar for small
values of $z$. However, the difference equation
is singular for the parameters of the first  and second examples. The numerical results are
illustrated in Figure~\ref{fig:FK1}. In Figures~\ref{fig:FK2} and \ref{fig:FK3}, we can
see an approximation to the value of $\Phi(P)$ near $z=0$. In each case, we provide
a graph of $\Phi(P^{\leq 100})$. These graphs  quantify the error in the approximation.

\end{example}

\subsection{The Heisenberg $XY$ model}\label{XYmodel}

Consider the difference equation mentioned in the introduction and
given by \eqref{eq:xy}. The characteristic polynomial
of the fixed point $\theta^*=0$ is
\(
\mathcal{F}(\lambda) = \lambda^2 -(2 +\varepsilon)\lambda+1.
\)
The corresponding parameterization equations can be written as
\begin{equation}\label{Heisenberg}
\sin (P(\lambda z) - P(z)) + \sin (P(\lambda^{-1} z) - P(z))
- \varepsilon\sin P(z) =0,
\end{equation}
where $\lambda$ is a stable root of $\mathcal{F}$.

Equating terms of order $n$ in \eqref{Heisenberg} we obtain that,
for $n=0$, the choice of $P_0=0$ corresponds to choosing the fixed point solution
we are studying. The
term of order $n=1$ amounts to choosing the manifold and setting the numerical scale
at which we are working.

The equations obtained matching order $n\ge2$ are
\begin{align}\label{XYorder n}
 \mathcal{F}(\lambda^{n+1}) P_{n+1} &=\left[ \sin \left(P^{\leq\,{n}} ( z) - P^{\leq\,{n}} (\lambda z)\right)\right]_{n+1}\\
 \nonumber
&+\left[ \sin \left(P^{\leq\,{n}} (z) - P^{\leq\,{n}} (\lambda^{-1} z)\right)\right]_{n+1}+
\varepsilon\left[ \sin \left(P^{\leq\,{n}} (z)\right)\right]_{n+1}.
\end{align}
The only difference with the Chirikov model or standard map is that we also have
to compute
$\bar S(z) = \sin (P(z) - P(\lambda z))$ and $\bar C(z) = \cos (P(z) - P(\lambda z))$.
Of course,
$\bar S(\lambda^{-1} z) = \sin (P(\lambda^{-1} z) - P(z))$ and $\bar C(\lambda^{-1} z) = \cos (P(\lambda^{-1} z) - P(z))$.
In the inductive step we assume that we know $P^{\leq\,(n-1)}$ and the
$S,C,\bar S,\bar C$ corresponding to
$P^{\leq\,(n-1)}$ up to order~$n$.
Using \eqref{XYorder n}, we can compute $P_n$ and then use \eqref{recursion} to compute $S,C,\bar S,\bar C$
corresponding to $P^{\leq\,n}$ up to order $n+1$ and the induction
can continue.

\subsection{Non resonant invariant manifolds in Froeschl\'{e} maps}
\label{sec:froeschle}

The Froeschl\'{e} map is a popular model
\cite{Froeschle72,Olvera94} of higher dimensional twist maps.
It is designed to be a model of the behavior of a double resonance.
 In the Lagrangian formulation, the Lagrangian of the model is a function
 $S:\mathbb{R}^2\times\mathbb{R}^2\to \mathbb{R}$ that is given by
 \[ S(\theta_0,\theta_1)=\frac12(\theta_1-\theta_1)^2+W(\theta_0), \]
 where $W:\mathbb{R}^2\to \mathbb{R}$
 is a potential functions such that $W(\theta+k)=W(\theta)$, for all
 $k\in\mathbb{Z}^2$.
 The resulting Euler-Lagrange
 equations are given by
 \begin{equation}\label{eq:froeschle}
 \theta_{k+1}-2\theta_{k}+\theta_{k-1}+\nabla W(\theta_{k})=0.
 \end{equation}
 If $\nabla W(0)=0$, then $\theta^*=0$ is a fixed point solution
 and the characteristic function is
 \[ \mathcal{L}(\lambda)=\det\left((\lambda+\lambda^{-1}-2)I+D^2W(0)\right). \]
 The particular example used by Froeschl\'{e} is
 \[ W(x_1,x_2)=a\cos(2\pi x_1)+b \cos(2\pi x_2)+c \cos(2\pi(x_1-x_2)). \]
 The matrix $I-\frac12D^2W(0)$ is a $2\times2$ symmetric matrix.
 Typically, it has two real eigenvalues $\omega_1$ and $\omega_2$.
 It turns out that there are four roots of $\mathcal{L}(\lambda)$,
 that constitute the spectrum of the fixed point. They are given by the
 solutions of \[ \lambda+\lambda^{-1}-2\omega_i=0. \] It is easy to see
 that these four solutions are given by $\omega_i\pm\sqrt{\omega_i^2-1}$,
 for $i=1,2$. From this, we conclude that the solutions are of the form $\lambda_1,\lambda_2,\lambda_1^{-1},\lambda_2^{-1}$, where
 $|\lambda_1|\leq|\lambda_2|\leq1$. We have to consider three possibilities:
 \begin{multicols}{3}
    \begin{enumerate}
 \item \label{casea}
 $0<\lambda_1\leq \lambda_2<1$.
 \item \label{caseb} $0<\lambda_1<1$, $|\lambda_2|=1$.
 \item \label{casec} $|\lambda_1|=|\lambda_2|=1$.
 \end{enumerate}
 \end{multicols}

The classical theory of invariant manifolds allows to associate an invariant
one dimensional manifold with $\lambda_1$ in cases
\ref{casea}) and \ref{caseb}), and a two dimensional invariant manifold
in case \ref{casea}). See also \cite{Llave97} for the case $0<\lambda_1=\lambda_2<1$.

The parameterization method allows also to make sense of each manifold
tangent to the space corresponding to $\lambda_2$ provided $\lambda_1^k\neq \lambda_2$ and
$\lambda_2^k\neq \lambda_1$, for $k\geq2, k\in\mathbb{N}$. The calculations
are remarkably similar to those of the stable manifolds for
the Chirikov  map studied before.

We now indicate the algorithm. Since we are considering the fixed point
$\theta^*=0$, we make $P_0=0$. As before, the
first coefficient of the parameterization is an eigenvector
associated with the eigenvalue $\lambda_2$. Again, we note that
the size of $P_1$ corresponds to different scales of the parameterization,
and does not affect the mathematical considerations. On the other hand,
choosing an appropriate scale is crucial in order to minimize round-off errors.

When $W$ is a trigonometric polynomial, as in the original
Froeschl\'{e} model, we can compute the components of
\(
\left[\nabla W\left(P^{\leq n-1}\right)\right]_n.
\)
We conclude that $n-$th coefficient of the parameterization satisfies
\begin{equation}\label{eq:recursion}
\left(\left(\lambda^n+\lambda^{-n}-2\right)I+D^2W(0)\right)P_n=
-\left[\nabla W\left(P^{\leq n-1}\right)\right]_n.
\end{equation}


\section{The continuously differentiable case}

The approach to invariant manifolds of this paper, also
applies when the map $Z$ is finite differentiable.
Many of the techniques presented above can be adapted 
to the finite differentiable case. Here, we will just present 
an analogue of Theorem~\ref{thm:main}, but we leave to the reader
other issues such as singular limits or dependence on parameters. 

\begin{theo}\label{finite-diff} 
Assume that $Z$ is $C^{r+1}$. 
Let 
 $\Lambda=\diag(\lambda_1,\ldots,\lambda_m)$,
where $\boldsymbol{\lambda}=(\lambda_1,\ldots,\lambda_m)$ is a stable
non-resonant vector. 
Assume 
that $B_0$ is invertible and 
that 
$r$ is sufficiently large, depending only on $\Lambda$. 
Then, we can find a $C^r$  $P$ solving. 
\eqref{eq:invariancelinear}, $P(0) = 0$  and such that the range of 
$P_1$ is the space corresponding to the eigenvalues in $\Lambda$. 
\end{theo} 

The conditions on  $r$ assumed in Theorem~\ref{finite-diff} 
will be made explicit in Lemma~\ref{lem:smooth}. 
It will be clear from the proof of 
Theorem~\ref{finite-diff}  that, if we assume more regularity
in $Z$, we can obtain differentiability with respect to parameters
(keeping $\Lambda$ fixed). 

We still study the equation $\Phi(P) = 0$
by implicit function methods but now $P$ ranges over a space of
finite differentiable functions. To apply the implicit function theorem,
we just need  the differentiability of the functional $\Phi$ and to
 show that $D\Phi(0)$ is invertible
with bounded inverse. One 
complication is that we cannot show 
 that $D\Phi(0)$ is boundedly invertible 
by matching powers. Nevertheless, we will show 
that $D\Phi(0)$ is invertible for 
for functions that vanish at 
high order. Hence, we will use that we 
can find power series solutions so that we can write
$P = P^< + P^\ge$ where $P^<$, is a polynomial we will assume known
and $P^\ge$ are functions vanishing to high order.

The functional analysis properties of $\Phi$ 
will be taken care of in the following lemmas. 

\begin{lemm}\label{Phidifferentiable} 
If $Z \in C^{r + \ell}$ $r, \ell \in \mathbb{N}$, then there
exists an open subset $\mathcal{U}\subset C^r(D,\mathbb{R}^d)$
such that
the operator $\Phi:\mathcal{U}\to C^r(D,\mathbb{R}^d)$  is $C^\ell$. 
Furthermore, if $\ell \ge 1$, we have 
\begin{equation} 
D\Phi(P) \varphi   = 
\sum_{j= 0}^N 
\partial_j Z\left( P, P\circ \Lambda, \ldots, P\circ  \Lambda^N \right) 
\varphi \circ \Lambda^j .
\end{equation} 
In particular, 
\begin{equation} 
\label{derivativeform} 
D\Phi(0) \varphi  = 
\sum_{j = 0}^n  B_j \left(\varphi\circ \Lambda^j \right)
\end{equation} 
where, as before,  $B_{i}=\partial_iZ(0,\ldots,0)$.
\end{lemm} 

The Lemma~\ref{Phidifferentiable} is a 
direct consequence of  the results of \cite{obaya99} on the composition
 operator. Of course, the formula \eqref{derivativeform} is easy to guess 
heuristically by noticing that it is the leading term in changes in $P$. 
We emphasize that we are considering $\Lambda$ fixed. The differentiability 
properties of the operator with respect to $\Lambda$ are much more subtle.

Let $D=\mathbb{R}^m(\rho)$, be
the closed disk, in $\mathbb{R}^m$, around the origin of radius $\rho$. As before, $\|\cdot\|_\infty$
is the uniform norm  defined in \eqref{eq:uniformnorm}.
To study $C^r$ spaces, it is convenient to recall that higher-order derivatives are 
multilinear maps. Given a $C^r$ function $g:\mathbb{R}^m\to\mathbb{R}^d$ and $k\leq r$,
then the $k-$th
derivative $D^kg(a)$ is a  $k-$linear map
$D^kg(a):\mathbb{R}^m\times\cdots\times\mathbb{R}^m\to\mathbb{R}^d$, symmetric under permutation of 
the order of the arguments. Recall also  that there is a natural norm
on the vector space of $k-$linear maps $B$ given by
\[
\|B\|=\sup\{\|B(\xi_1,\ldots,\xi_{k})\|_\infty:\|\xi_1\|_\infty=\cdots=\|\xi_{k}\|_\infty=1\}.
\]
Once we fix a system of coordinates, $D^kg$ can be expressed in terms of 
the partial derivatives
\[
 \frac{1}{k!}D^kg(a)(h,\ldots,h)=\sum_{|\alpha|=k}\frac{1}{\alpha!}\partial^\alpha g(a)\,h^\alpha.
\]
The usual norm in $C^r(D,\mathbb{R}^d)$ spaces is $\|g\|_{C^r} = \max_{i \le r} \sup_{x\in D} \| D^i g(x)\|$.

To study our problem, 
we define the closed subspace of $C^r(D,\mathbb{R}^d)$. 
\[
   H_r=\left\{P \in C^r(D,\mathbb{R}^d):  D^i P(0)=0 \quad 0 \le i \le r-1\right\}.
\]
Clearly $P\in H_r$ if and only if $\partial^\alpha P(0)=0$, for all $ |\alpha|\leq r-1 $.
The space $H_r$ is a Banach space if endowed  with the norm 
$\|P\|'_r = \sup\left\{ \| D^r P(x)\|:x\in D\right\}$.
Because $P$ vanishes with its derivatives at $0$ and we are considering a bounded 
domain, this norm is equivalent to 
the standard $C^r$ norm.

Assume that $B_0$ is non-singular. 
Then, we can rewrite \eqref{derivativeform} as:
\begin{equation}\label{derivativegrouped} 
D\Phi(0)\varphi  = 
B_0\left[ \varphi + \sum_{j=1}^N B_0^{-1} B_j \left(\varphi \circ \Lambda^j\right)\right] 
=  B_0\left( \Id +  L\right)\varphi,
\end{equation} 
where $L$ is the operator defined by:
\[
L(\varphi)=\sum_{i=1}^NB_0^{-1}B_{j} \,\left(\varphi\circ  \Lambda^{j}\right).
\]

\begin{lemm}\label{lem:smooth}
    Let $\mu$ be a real number such that:
    \begin{enumerate}
     \item $\|\Lambda\|^r\leq \mu<1$,
      \item $\displaystyle \sum_{i=1}^N\mu^i\left\|B_0^{-1}B_{i}\right\|<1$.
    \end{enumerate}
    Then the linear operator $D\Phi(0):H_r\to H_r$ is invertible with bounded inverse.
\end{lemm}

We emphasize that if $\boldsymbol{\lambda}$ is a stable non-resonant vector of eigenvalues, 
the conditions of Lemma~\ref{lem:smooth} are satisfied for sufficiently large $r$. 
This is the condition alluded to in Theorem~\ref{finite-diff}.

\begin{proof}
Since $D\Phi(0) = B_0^{-1}(\Id +L) $ 
it will suffice to prove that $L$ is a contraction
in the $\| \cdot\|'_r $ norm. 
We note  that if $\varphi \in H_r$, then 
$\varphi\circ \Lambda^j\in H_r$ and
 $L \varphi \in H_r$. 
In addition,
\[
\left[D^r\left(\varphi\circ \Lambda^j\right)(x)\right](\xi_1,\ldots,\xi_r)=
\left[\left(D^r\varphi\right)(\Lambda^jx)\right](\Lambda^j\xi_1,\ldots,\Lambda^j\xi_r).
\]
Hence, if $x\in D$ then
\[
 \left\|D^r\left(\varphi\circ \Lambda^j\right)(x) \right\|\leq 
 \left\| \left(D^r\varphi\right)(\Lambda^jx)\right\|\left\|\Lambda^j\right\|^r\leq
\left\| \left(D^r\varphi\right)(\Lambda^jx)\right\|\mu^j,
\]
and, since $\Lambda D\subset D$, we get
that $\|\varphi\circ\Lambda^j\|'_r\leq \mu^j\|\varphi\|'_r$.
From this we conclude 
\[ 
\| L (\varphi) \|'_r   \leq  \sum_{j=1}^N \left\|B_0^{-1}B_{i}\right\|\|\varphi \circ \Lambda^j\|'_r 
       \leq \left(\sum_{i=1}^N\mu^i\left\|B_0^{-1}B_{i}\right\|\right) \|\varphi\|'_r.
\]
This shows that $L$ is a contraction.
\end{proof}

To complete the proof of 
Theorem ~\ref{finite-diff}, 
we argue in a similar manner as 
in the applications of the implicit function theorem before.
Following the method in Section~\ref{higherorder}, and fixing
 $P_1$ to be an embedding on 
the space, we can find a unique polynomial $P^<$ 
of degree ${r-1}$ in such a way that $P^<(0) = 0$, $(P^<)'(0) = P_1$
and 
\(
D^j \Phi(P^<)(0) = 0,
\)
for all $0 \le j \le r-1$.
We recall that the coefficients of $P^<$ are obtained by matching coefficients.
In addition,
if we choose $P_1$ small, they will also be small. 

Furthermore, we also note that if $P^\ge \in H_r$, we have
that 
\[
D^j \Phi\left(P^< + P^\ge \right)(0) = 0; \quad 0 \le j \le r-1.
\]
The above remarks can be formulated in terms of functional analysis as saying 
that the operator 
$\tilde \Phi( P^\ge) = \Phi( P^< + P^\ge) $ maps 
an open subset of $H_r$ into $H_r$. 

If we take $P_1$ small without changing its range (this is the same as 
the change of scale that we considered before), we have that 
$\tilde \Phi(0)$ will be as small as desired, $D\tilde \Phi(0) = D\Phi(P^<)$ will 
approach $D\Phi(0)$ and, in particular will  be invertible. The differentiability properties
of $\tilde \Phi$ will remain uniformly differentiable. Hence, we can 
deduce Theorem~\ref{finite-diff} from the implicit function theorem in Banach spaces.

More explicitly, consider the fixed point 
problem for $P^\ge$ given by: 
\[
-D\Phi(0)^{-1} \Phi( P^< +  P^\ge) + P^\ge = P^\ge,
\]
as acting on $H_r$. When $P^<$ is chosen corresponding to a specific small $P_1$,
the left hand side will be a contraction and map a ball in $H_r$ near the origin into itself.


\section{Final comments}

We have studied invariant objects in the context of difference equations.
The goal was to find the analog of stable and unstable manifolds that
are associated to fixed point solutions of hyperbolic type. We generalized the
method proposed in \cite{Cabre2003i, Cabre2003ii, Cabre2005iii}
and used the implicit function theorem in Banach spaces to
perform this task.

The method is not only theoretical but can be implemented numerically
as it was shown in
example \ref{exa:xxx}. The method is robust in the sense that it can
deal with certain singular limits. In particular, slow manifolds
can be detected and approximated even in the presence of a singularity.

For future work we plan to consider the finite differentiable case.
Also we can use the tools to consider the case of conformally symplectic maps
and applications to economics. In addition, we showed that the
variational theory has a smooth dependence on parameters and is robust. 
We plan to explore variational formulations of Melnikov's theory.

\appendix
\bigskip

\section{Banach spaces of analytic functions}\label{sec:banach}

In this appendix we study spaces of analytic functions
(taking values and having range in Banach spaces) and,
in particular, the composition operator
$\mathcal{C}_f(g)= f \circ g $ between analytic
functions. We show that the composition operator is itself
analytic when defined in spaces of analytic functions.

We call attention to the paper \cite{Meyer75} which carried out
a similar study and showed that the operator $\Gamma(f, g) = f \circ g $ was $C^\infty$.
The paper \cite{Meyer75} showed also that many problems in
the theory of dynamical systems --invariant manifolds,
limit cycles, conjugacies-- could be reduced to problems involving the
composition operator. Using the result of analyticity presented here, most of the
regularity results in \cite{Meyer75} can be improved from $C^\infty$ to
analytic.

There are of course, specialized books which contain much more
material than we need. For example \cite{Hoffman,Nachbin}.

\subsection{Analytic functions in general Banach spaces}
Let $E$ and $F$ be Banach spaces.
We define $\mathcal{S}_{k}(E,F)$ as the linear space of bounded symmetric $k-$linear
functions from $E$ to $F$. For each $a_{k}\in \mathcal{S}_{k}(E,F)$,
the notation $a_{k}\left( x^{\otimes k}\right)$ denotes the
$k-$homogeneous function $E\to F$ given by
\[
a_{k}\left( x^{\otimes k}\right)=a_{k}\left( x,\ldots,x\right).
\]
On $\mathcal{S}_{k}(E,F)$, we require to have a norm $\|\cdot\|_{\mathcal{S}_{k}(E,F)}$ such that
\begin{equation}\label{def:des}
\|a_{k}(x_1,\ldots,x_{k})\|_F\leq\|a_{k}\|_{\mathcal{S}_{k}(E,F)}\|x_1\|_E\cdots\|x_{k}\|_E,
\end{equation}
for all $a_{k}\in\mathcal{S}_{k}(E,F)$.
In particular, this implies that
\[
\left\|a_{k}\left( x^{\otimes k}\right)\right\|_F\leq \|a_{k}\|_{\mathcal{S}_{k}(E,F)}\|x\|_E^k.
\]

\begin{remark}
 Of course, a natural choice of norms is
\[
\|a_{k}\|=\sup\{\|a_{k}(x_1,\ldots,x_{k})\|_\delta:\|x_1\|_E=\cdots=\|x_{k}\|_E=1\},
\]
but there are others. In the specific case $E=\mathbb{C}^\ell$ and $F=\mathbb{C}^d$,
we have a found that the norm defined in
\eqref{def:normsym} is useful and well adapted to analytic functions of complex variables.
\end{remark}

Once the norms on the spaces $\mathcal{S}_{k}(E,F)$ are fixed, we can define analytic functions
from $E$ to $F$. Throughout the rest of the appendix,
$E(\delta)=\{\|x\|_E\leq\delta\}$ will denote the closed ball in $E$ centered at the origin with radius $\delta$.
\begin{defi}\label{def:analytic}
Let $\delta>0$.
The space of analytic functions from $E$ to $F$ with radius of convergence $\delta$ is
the set of functions $A_\delta(E,F)$ given by
 \[
A_\delta(E,F)=\left\{f:E(\delta)\to F\left|f(x)=\sum_{k=0}^\infty a_{k}\left( x^{\otimes k}\right),
a_{k}\in \mathcal{S}_{k}(E,F), \|f\|_\delta<\infty\right.\right\},
\]
where $\|f\|_\delta:=\sum_{k=0}^\infty\|a_{k}\|_{\mathcal{S}_{k}(E,F)}\,\delta^k$. The a norm $\|\cdot\|_\delta$ makes
$A_\delta(E,F)$ into a Banach space. Given $\rho>0$, we will
denote $A_\delta^\rho(E,F)$ the closed ball centered at the origin with radius $\rho$.
\end{defi}

It is possible to check that if $g\in A_\delta^\rho(E,F)$ and
$f\in A_\rho(F,G)$, then $f\circ g\in A_\delta(E,G)$. In fact, it is
proved in \cite{Meyer75} that the function
$\Gamma:A_\delta^\rho(E,F)\times A_\rho(F,G)\to A_\delta(E,G)$ given by
$\Gamma(g,f)=f\circ g$ is $C^\infty$. In particular, this is true if we consider the Banach spaces of analytic functions
that are used in the text. Further details are given below for some specific situations that are relevant to our
problem. In subsection \ref{sec:comp},
we will show that an operator similar to $\Gamma$ is analytic.

\subsection{Analytic functions of complex variables}

We are interested in some concrete Banach spaces of analytic functions.
We will define specific norms for the following spaces of
$k-$linear functions:
\begin{itemize}
 \item $\mathcal{S}_{k}(\mathbb{C}^\ell,\mathbb{C}^d)$,
 \item $\mathcal{S}_{k}(E,F)$ where $E=A_\delta(\mathbb{C}^m,\mathbb{C}^\ell)$ and $F=A_\delta(\mathbb{C}^m,\mathbb{C}^d)$.
\end{itemize}
First, consider the Euclidean spaces $\mathbb{C}^r$ with uniform norm
\begin{equation}\label{eq:uniformnorm}
 \|\eta\|_\infty=\|(\eta_1,\ldots,\eta_r)\|_\infty=\max\{|\eta_1|,\ldots,|\eta_r|\},
\end{equation}
for each $\eta=(\eta_1,\ldots,\eta_r)\in\mathbb{C}^r$.
We notice that this norm satisfies
\begin{equation}\label{eq:desmult}
 |z^\alpha|\leq\left(\|z\|_\infty\right)^{|\alpha|},
\end{equation}
for all multi-indices $\alpha \in\mathbb{Z}^r_+$, and $z\in\mathbb{C}^r$.

 Let $\mathcal{S}_{k}(\mathbb{C}^\ell,\mathbb{C}^d)$ be the space
 of bounded symmetric $k-$linear functions from $\mathbb{C}^\ell$ to $\mathbb{C}^d$.
 Clearly, the spaces $\mathcal{S}_{k}(\mathbb{C}^\ell,\mathbb{C}^d)$ consist of
 polynomial functions in $\ell$ complex variables and $d$ coordinates.
 It is well known that $\mathcal{S}_{k}(\mathbb{C}^\ell,\mathbb{C}^d)$
 and the space $\mathcal{H}_{k}(\mathbb{C}^\ell,\mathbb{C}^d)$
 of homogeneous polynomials of degree $k$ are linearly isomorphic.

 In particular, given $q\in \mathcal{S}_{k}(\mathbb{C}^\ell,\mathbb{C}^d)$,
 the polynomial $ {q}\left( z^{\otimes k}\right)$ is an element of
 $\mathcal{H}_{k}(\mathbb{C}^\ell,\mathbb{C}^d)$.
 Conversely, for each homogeneous polynomial
 $g\in\mathcal{H}_{k}(\mathbb{C}^\ell,\mathbb{C}^d)$, there exists an
 unique $q\in \mathcal{S}_{k}(\mathbb{C}^\ell,\mathbb{C}^d)$ such
 that $ {q}\left( z^{\otimes k}\right)=g(z)$. Below, in the proof of
 Lemma \ref{lem:hhh}, we will describe how such $q$ can be found.

We need to provide each $\mathcal{S}_{k}(\mathbb{C}^\ell,\mathbb{C}^d)$
with a norm that satisfies condition \eqref{def:des}.
Consider ${q}\in \mathcal{S}_{k}(\mathbb{C}^\ell,\mathbb{C}^d)$. Clearly, the
homogeneous function $ {q}\left( z^{\otimes k}\right)$ is of the form
\begin{equation}\label{eq:homo}
 {q}\left( z^{\otimes k}\right)=\sum_{|\alpha|=k}z^\alpha\eta_\alpha,
\end{equation}
where $\eta_\alpha$ are constant vectors of coefficients in $\mathbb{C}^d$.
Using these coefficients, we define the norm of as
\begin{equation}\label{def:normsym}
\|{q}\|_{\mathcal{S}_{k}(\mathbb{C}^\ell,\mathbb{C}^d)}:=\sum_{|\alpha|=k}\|\eta_\alpha\|_\infty.
\end{equation}
If it is clear from the context, we will write the previous norm as
simply $\|{q}\|$.

\begin{lemm}\label{lem:hhh}
The norm on $\mathcal{S}_{k}(\mathbb{C}^\ell,\mathbb{C}^d)$ defined
in \eqref{def:normsym} satisfies
condition \eqref{def:des}.
\end{lemm}
\proof
Fixing $k>0$, we consider a multi-index $\alpha \in\mathbb{Z}^\ell_+$ with $|\alpha|=k$.
We will give an useful description of the unique multi-linear symmetric polynomial function
$\kappa_\alpha\in \mathcal{S}_{k}(\mathbb{C}^\ell,\mathbb{C})$ that satisfies $\kappa_\alpha(z^{\otimes k})=z^\alpha$.

Let $\Pi_{k}$ be the permutation group of $k$ elements.
Let $\beta_1,\ldots,\beta_{k}\in\mathbb{Z}^\ell_+$ be multi-indices such that $\beta_1+\cdots+\beta_{k}=\alpha$ and $|\beta_i|=1$.
Using these vectors, we define the function $\kappa_\alpha\in \mathcal{S}_{k}(\mathbb{C}^\ell,\mathbb{C})$ as
\begin{equation}\label{def:kappa}
\kappa_\alpha(z_1,\ldots,z_{k}):=\frac1{k!}\sum_{\pi\in\Pi_{k}}
z_{\pi(1)}^{\beta_1}z_{\pi(2)}^{\beta_2}\cdots z_{\pi(k)}^{\beta_{k}},
\end{equation}
where $z_1,\ldots,z_{k}\in\mathbb{C}^\ell$.
It is clear that $\kappa_\alpha$ is a well-defined multi-linear function and 
$\kappa_\alpha\left( z^{\otimes k}\right)=\kappa_\alpha(z,z,\ldots,z)=z^\alpha$.
In addition, using \eqref{eq:desmult}, we get that $\kappa_\alpha$ satisfies
\begin{align*}
 |\kappa_\alpha(z_1,\ldots,z_{k})|&\leq
 \frac1{k!}\sum_{\pi\in\Pi_{k}}|z_{\pi(1)}^{\beta_1}||z_{\pi(2)}^{\beta_2}|\cdots |z_{\pi(k)}^{\beta_{k}}|\\
 & \leq
 \frac1{k!}\sum_{\pi\in\Pi_{k}}\|z_{\pi(1)}\|_\infty\|z_{\pi(2)}\|_\infty\cdots \|z_{\pi(k)}\|_\infty =
 \|z_1\|_\infty\cdots \|z_{k}\|_\infty.
\end{align*}
From this, we conclude that if $q\in\mathcal{S}_{k}(\mathbb{C}^\ell,\mathbb{C}^d)$ and
$q$ satisfies \eqref{eq:homo} then
\[
{q}\left(z_1,\ldots,z_{k} \right)=\sum_{|\alpha|=k}\kappa_\alpha(z_1,\ldots,z_{k})\eta_\alpha,
\]
and therefore
\[
\|{q}\left(z_1,\ldots,z_{k}\right)\|_\infty\leq
\sum_{|\alpha|=k}|\kappa_\alpha(z_1,\ldots,z_{k})|\|\eta_\alpha\|_\infty
\leq \|{q}\|_{\mathcal{S}_{k}(\mathbb{C}^\ell,\mathbb{C}^d)}\, \|z_1\|_\infty\cdots \|z_{k}\|_\infty.
\]
The last inequality implies that the norm defined in
\eqref{def:normsym} satisfies \eqref{def:des}.
\qed

From the discussion above, we also conclude that if
$g\in A_\rho(\mathbb{C}^{\ell},\mathbb{C}^{d})$ then $g$ is of the
form
\begin{equation}\label{eqn:hash2}
g(z)=\sum_{k=0}^\infty
\sum_{|\alpha|=k}z^\alpha\eta_\alpha.
\end{equation}
Therefore, with the norms defined in \eqref{def:normsym} on the space of
$k-$symmetric functions $\mathcal{S}_{k}(\mathbb{C}^{\ell},\mathbb{C}^{d})$, we
get the following norm, for any $g\in A_\rho(\mathbb{C}^{\ell},\mathbb{C}^{d})$:
\begin{equation}\label{eqn:hashash}
 \|g\|_{\rho}=\sum_{k=0}^\infty
\left(\sum_{|\alpha|=k}\|\eta_\alpha\|_\infty\right)\rho^k.
\end{equation}
This norm has the following property. If $g\in A_{\rho}(\mathbb{C}^{\ell},\mathbb{C})$
and $\eta\in \mathbb{C}^{d}$ then $g\cdot \eta\in A_{\rho}(\mathbb{C}^{\ell},\mathbb{C}^d)$ and
\begin{equation}\label{eqn:prop}
 \|g\cdot \eta\|_{\rho}=\|g\|_{\rho}\|\eta\|_{\infty}.
\end{equation}
We also have the following.

\begin{lemm}
 If $g_1,g_2\in A_{\rho}(\mathbb{C}^{r},\mathbb{C})$, then
$g_1\,g_2\in A_{\rho}(\mathbb{C}^{r},\mathbb{C})$ and
\[\|g_1g_2\|_{\rho}\leq\|g_1\|_{\rho}\|g_2\|_{\rho}.\]
\end{lemm}
\begin{proof}
 Let $g_1,g_2\in A_{\rho}(\mathbb{C}^{r},\mathbb{C})$. Then, they are of
 the form
 \[
 g_i(z)=\sum_{k=0}^\infty
\sum_{|\alpha|=k}z^\alpha\eta_\alpha^i,
 \]
 for $i=1,2$. This implies that
 \begin{align*}
 g_1(z)g_2(z)&=\sum_{\ell_1,\ell_2=0}^\infty
\sum_{|\alpha|=\ell_1}
\sum_{|\beta|=\ell_2}z^{\alpha+\beta}\eta_\alpha^1\eta_\beta^2\\
&=
\sum_{\ell=0}^\infty
\sum_{|\alpha|+|\beta|=\ell}z^{\alpha+\beta}\eta_\alpha^1\eta_\beta^2
=\sum_{\ell=0}^\infty\sum_{|\gamma|=\ell} z^{\gamma}
\left(
\sum_{\alpha+\beta=\gamma}\eta_\alpha^1\eta_\beta^2
\right),
 \end{align*}
 for all
 $z\in\mathbb{C}^{r}(\rho)$.
 Notice that the coefficients $\eta_\alpha^i$ are complex numbers
 and therefore they satisfy $\left|\eta_\alpha^1\,\eta_\beta^2
\right|= \left|\eta_\alpha^1\right|\left|\eta_\beta^2
\right|$.
 Using \eqref{eqn:hashash}, we get that
 \begin{align*}
 \|g_1\,g_2\|_\rho&=
 \sum_{\ell=0}^\infty\sum_{|\gamma|=\ell}
\left|
\sum_{\alpha+\beta=\gamma}\eta_\alpha^1\eta_\beta^2
\right| \rho^{\ell}
\leq \sum_{\ell=0}^\infty\sum_{|\gamma|=\ell}
\sum_{\alpha+\beta=\gamma}\left|\eta_\alpha^1\right|\left|\eta_\beta^2
\right| \rho^{\ell}\\ &=
\left(\sum_{\ell_1=0}^\infty
\sum_{|\alpha|=\ell_1}\left|\eta_\alpha^1\right|\rho^{\ell_1}\right)
\left(\sum_{\ell_2=0}^\infty
\sum_{|\beta|=\ell_1}\left|\eta_\beta^2\right|\rho^{\ell_2}\right)=
\|g_1\|_{\rho}\|g_2\|_{\rho}.
 \end{align*}
 This also shows that the series involved in $ \|g_1\,g_2\|_\rho$ converges.

\end{proof}

The previous result also shows that $A_{\rho}(\mathbb{C}^{r},\mathbb{C})$ is a Banach algebra.
Furthermore, if $g\in A_{\rho}(\mathbb{C}^{r},\mathbb{C}^{s})$ and
$\alpha \in\mathbb{Z}^s_+$ is
a multi-index, then $g^\alpha\in A_{\rho}(\mathbb{C}^{r},\mathbb{C})$ and
\begin{equation}\label{desaux}
\|g^\alpha\|_{\rho}\leq\|g\|_{\rho}^{|\alpha|}.
\end{equation}

\subsection{The composition operator}\label{sec:comp}

Let $E=A_\delta(\mathbb{C}^m,\mathbb{C}^\ell)$ and $F=A_\delta(\mathbb{C}^m,\mathbb{C}^d)$ be two Banach spaces of analytic
functions.
We will define norms on the corresponding spaces of $k-$symmetric linear and bounded functions.
For each $b_{k}\in\mathcal{S}_{k}\left(E, F\right)$
we define the norm
\begin{equation}\label{eq:more}
 \|b_{k}\|_{\mathcal{S}_{k}\left(E, F\right)}:=\sup\{\|b_{k}(g_1,\ldots,g_{k})\|_\delta:\|g_1\|_\delta=\cdots=\|g_{k}\|_\delta=1\}.
\end{equation}
Notice that this norm satisfies condition \eqref{def:des}.
With these norms, we can show that each analytic function in $A_{\rho}(\mathbb{C}^\ell,\mathbb{C}^d)$
can be extended to an analytic operator in $A_{\rho}(E, F)$, acting on spaces of analytic functions.

\begin{lemm}\label{lemabonito}
Let $f\in A_{\rho}(\mathbb{C}^\ell,\mathbb{C}^d)$ and $\delta>0$.
Let $E=A_\delta(\mathbb{C}^m,\mathbb{C}^\ell)$ and $F=A_\delta(\mathbb{C}^m,\mathbb{C}^d)$ be two Banach spaces of analytic
functions.
If $\mathcal{C}_f:E(\rho)\to F$ is the operator given by
$\mathcal{C}_f(g)=f\circ g$, then $\mathcal{C}_f$ is well defined and analytic.\footnote{This means that
$\mathcal{C}_f\in A_{\rho}(E, F)$.}
In addition, with the norms on the spaces
$\mathcal{S}_{k}(E,F)$ defined in \eqref{eq:more}, we have that
$\|\mathcal{C}_f\|_{A_{\rho}(E, F)}\leq\|f\|_\rho$.
\end{lemm}

\proof

We know that $f$ is of the form
\[
f(z)=\sum_{k=0}^\infty a_{k}\left( z^{\otimes k}\right),
\]
with \(
 a_{k}\in \mathcal{S}_{k}(\mathbb{C}^\ell,\mathbb{C}^d),
 \)
and
\[
\|f\|_{\rho}=\sum_{k=0}^\infty\|a_{k}\|_{\mathcal{S}_{k}(\mathbb{C}^\ell,\mathbb{C}^d)}\,{\rho}^k<\infty.
\]
Each multi-linear function $a_{k}\in \mathcal{S}_{k}(\mathbb{C}^\ell,\mathbb{C}^d)$
can be extended to a function $\tilde a_{k}\in \mathcal{S}_{k}\left(E, F\right)$ by
\[
 \left[\tilde a_{k}(g_1,\ldots,g_{k})\right](z):= a_{k}(g_1(z),\ldots,g_{k}(z)),
\]
for each $z\in\mathbb{C}^m$ such that $\|z\|_\infty\leq\delta$, and all
$g_1,\ldots,g_{k}\in E$.

As in the proof of Lemma~\ref{lem:hhh}, we can assume that each $a_k$ is of the form
\[
{a_k}\left(z_1,\ldots,z_{k} \right)=\sum_{|\alpha|=k}\kappa_\alpha(z_1,\ldots,z_{k})\eta_\alpha.
\]
In particular, each polynomial $\kappa_\alpha$ can be extended to a function
 $\tilde \kappa_\alpha\in \mathcal{S}_{k}\left(E, F\right)$. This implies that
 \[
\tilde a_{k}\left(g_1,\ldots,g_{k} \right)=\sum_{|\alpha|=k}\tilde \kappa_\alpha(g_1,\ldots,g_{k})\eta_\alpha.
\]
In addition, using \eqref{def:kappa} and
inequality \eqref{desaux}, we get that
\[
\|\tilde \kappa_\alpha(g_1,\ldots,g_{k})\|_\delta\leq \|g_1\|_\delta\cdots \|g_{k}\|_\delta.
\]
Using \eqref{eqn:prop},
\[
\|\tilde a_{k}\left(g_1,\ldots,g_{k} \right)\|_\delta\leq
\left(\sum_{|\alpha|=k}\|\eta_\alpha\|_\infty\right)\|g_1\|_\delta\cdots \|g_{k}\|_\delta.
\]
Therefore, for all $k\geq0$, the norm \eqref{eq:more} on $\mathcal{S}_{k}(E,F)$
is such that
\[
\|\tilde a_{k}\|_{\mathcal{S}_{k}\left(E, F\right)}\leq \|a_{k}\|_{\mathcal{S}_{k}(\mathbb{C}^\ell,\mathbb{C}^d)}.
\]
Finally, we notice that the operator $\mathcal{C}_f$ can be written as
\[
\mathcal{C}_f(g)=\sum_{k=0}^\infty\sum_{|\alpha|=k}\tilde a_{k}\left(g^{\otimes k}\right),
\]
and conclude that
\[\|\mathcal{C}_f\|_{A_{\rho}(E, F)}=
\sum_{k=0}^\infty
\|\tilde a_{k}\|_{\mathcal{S}_{k}\left(E, F\right)}\,\rho^k
\leq \sum_{k=0}^\infty
\|a_{k}\|_{\mathcal{S}_{k}(\mathbb{C}^\ell,\mathbb{C}^d)}\,{\rho}^k=\|f\|_\rho<\infty.
\]
\qed

\bibliographystyle{alpha}
\bibliography{tev0}
\clearpage

\end{document}